\documentclass[12pt, a4paper]{amsart}

\usepackage{color}
\usepackage[english]{babel}
\usepackage{amsmath,amsfonts,amssymb,amsthm,amstext,amscd,amsbsy,mathrsfs}
\usepackage{mathtools}
\usepackage{indentfirst}
\usepackage{hyperref}
\usepackage[top=26mm,right=26mm,bottom=26mm,left=26mm]{geometry}

\makeatletter
\renewenvironment{thebibliography}[1]{%
	\section*{References}%
	\@mkboth{REFERENCES}{REFERENCES}%
	\normalfont\bibliofont          
	\list{\@biblabel{\@arabic\c@enumiv}}%
	{\settowidth\labelwidth{\@biblabel{#1}}%
		\leftmargin\labelwidth
		\advance\leftmargin\labelsep
		\itemindent\z@
		\listparindent\z@
		\usecounter{enumiv}%
		\let\p@enumiv\@empty
		\renewcommand\theenumiv{\@arabic\c@enumiv}}%
	\sloppy
	\clubpenalty4000
	\@clubpenalty\clubpenalty
	\widowpenalty4000%
	\sfcode`\.\@m}
{\endlist}
\makeatother

\def\D{\mathcal{D}}
\def\B{\mathcal{B}}
\def\P{\mathcal{P}}
\def\Aut{{\rm Aut}}
\def\Sym{{\rm Sym}}
\def\ZZ{{\mathbb Z}}
\def\DD{{\rm D}}
\def\SS{{\rm S}}
\def\AA{{\rm A}}

\def\PSL{{\rm PSL}}
\def\PGL{{\rm PGL}}
\def\PGU{{\rm PGU}}
\def\PGammaL{{\rm P\Gamma L}}
\def\PG{{\rm PG}}
\def\PGammaL{{\rm P\Gamma L}}
\def\AGammaL{{\rm A\Gamma L}}
\def\AGL{{\rm AGL}}
\def\GU{{\rm GU}}

\def\GF{\mathbb{F}}
\def\GL{{\rm GL}}
\def\GammaL{{\rm \Gamma L}}
\def\Sz{{\rm Sz}}
\def\Out{{\rm Out}}

\newtheorem{theorem}{Theorem}[section]
\newtheorem{proposition}[theorem]{Proposition}
\newtheorem{lemma}[theorem]{Lemma}
\newtheorem{corollary}[theorem]{Corollary}

\newtheorem{remark}{Remark}
\newtheorem{hypothesis}[theorem]{Hypothesis}

\date{}

\begin{document}
\title[]{Locally dihedral block designs and primitive groups with dihedral point stabilizers}

\author{Jianfu Chen}
	\address{School of Mathematics and Computational Science\\
		Wuyi University,  Jiangmen 529030,  P.R.China}
	\email{chenjf@wyu.edu.cn; jmchenjianfu@126.com}

\author{Yanni Wu}
	\address{Department of Mathematics \\
		Southern University of Science and Technology,  Shenzhen 518055,  P.R.China}
	\email{12031209@mail.sustech.edu.cn}

\author{Binzhou Xia}
	\address{School of Mathematics and Statistics \\
		The University of Melbourne,  Parkville, VIC 3010, Australia}
	\email{binzhoux@unimelb.edu.au}

\begin{abstract}
Let $\D$ be a block design admitting a locally transitive automorphism group $G$.
We say that $\D$ is $G$-point-locally dihedral if the induced local action $G_x^{\D(x)}$ is dihedral for each point $x$, and that $\D$ is $G$-block-locally dihedral if the induced local action $G_B^B$ is dihedral for each block $B$. If both conditions hold, $\D$ is called $G$-locally dihedral.
We give a classification of primitive permutation groups with dihedral point stabilizers and apply this to classify point-locally dihedral block designs.
In particular, for symmetric designs with a dihedral or abelian local action, we show that $G_x$ and $G_B$ are conjugate in $G$, and that either $G$ acts imprimitively on both points and blocks, or $G$ is a Frobenius group of odd order.

\vskip0.1in


\vskip0.1in

\end{abstract}

\maketitle


\section{Introduction}

A block design, or a $2$-$(v,k,\lambda)$ design, is an incidence structure $\mathcal{D}=(\mathcal{P}, \mathcal{B})$, where $\mathcal{P}$ is a set of $v$ points, $\mathcal{B}$ is a set of $b$ $k$-subsets of $\mathcal{P}$, such that any 2-subset of $\P$ is contained in exactly $\lambda$ blocks.
The cardinality of the set $\D(x)$ of blocks through a point $x\in\P$ is a constant denoted by $r$.
A symmetric design is a 2-design that satisfies $v=b$.
For a symmetric design, it is known that any two blocks are incident with exactly $\lambda$ points.
In this paper, all 2-designs are considered non-trivial, meaning they satisfy $2<k<v-1$.
An automorphism group $G$ of a 2-design $\mathcal{D}$ is a permutation group of the point set $\mathcal{P}$ that preserves the block set $\mathcal{B}$.
The set of all automorphisms of $\D$ forms a group $\Aut(\D)$, called the full automorphism group of $\D$.
If $G$ acts transitively on the set of incident point-block pairs of $\D$, then $G$ is called $G$-flag-transitive.
For more basic facts on $2$-designs, see \cite[Section 2.1]{Dembowski}, \cite[Chapter II]{Handbook}, \cite[Chapter 3]{Biggs}, and \cite{BethDesign} for examples.

Classifying block designs admitting an automorphism group acting transitively on both points and blocks is a long-term research project.
For instance, the classification of flag-transitive finite linear spaces, namely, $2$-$(v,k,1)$ designs, was accomplished by Kantor \cite{KantorProplane} and a team of six authors~\cite{Buekenhout2vk1}.
Symmetric designs admitting certain automorphism groups were studied in \cite{Kantor2tr} and \cite{DempwolffSymaffine3, DempwolffSymAlmost3}, respectively for 2-transitive groups and primitive groups of rank three.
In recent years, there has been a particular interest in studying flag-transitive designs with certain parameters.

We inspect block designs that admit an automorphism group satisfying certain local properties.
Specifically, we study automorphism groups by restricting the point stabilizers (or block stabilizers), or their actions on the set of adjacent blocks (or adjacent points) to be some certain types of groups.
A natural approach is to consider locally transitive automorphism groups, namely, those $G$ with both $G_x^{\D(x)}$ and $G_B^B$ transitive, where $x\in\P$ and $B\in \B$.
It is straightforward to verify that the local transitivity is equivalent to the flag-transitivity (see \cite[Lemma 5.1]{LocallyPrimitive}).
Further, an automorphism group is called locally primitive if both the local actions $G_x^{\D(x)}$ and $G_B^B$ are primitive.
In 2025, Chen, Hua, Li and Wu \cite{LocallyPrimitive} studied locally primitive block designs and reduced such groups to affine and almost simple types.
Very recently, they further classified locally 2-homogenous block designs \cite{Locally2homo}.

The present paper studies the locally transitive automorphism groups of block designs whose local action $G_x^{\D(x)}$ or $G_B^B$ is dihedral.
In addition, we obtain some results concerning symmetric designs with an abelian local action.
For a block design $\D$ admitting a locally transitive automorphism group $G$, $G$ is called point-locally dihedral if the local action $G_x^{\D(x)}$ is dihedral; $G$ is called block-locally dihedral if the local action $G_B^B$ is dihedral. If $G$ is both point-locally dihedral and block-locally dihedral, then $G$ is said to be locally dihedral. The definition of locally abelian is analogous.

As preliminary work on studying block designs with a dihedral local action, Section \ref{SecPriDihedral} provides an explicit classification of almost simple primitive groups with dihedral point stabilizers.
This extends earlier results of Li \cite{LicaihengD2n}, Quirin \cite{QuirinSmallOrbitals}, and Wang \cite{WangJiesuborbitun2p}.
They previously studied the primitive groups with dihedral point stabilizers but did not obtain a complete classification.
Note that, in 2011, Li and Zhang \cite{ZhanghuaLicaiheng} studied the primitive permutation groups with solvable point stabilizers, which is a more general case. However, the explicit list of groups with dihedral point stabilizers cannot be directly read off from their classification.
In Theorem \ref{ProLCH}, we obtain an explicit and complete classification by further analyzing the actions of the outer automorphism groups of certain simple groups.
In addition, we prove that in $\PGL_2(q)$, the intersection of two maximal subgroups respectively of type $\DD_{2(q-1)}$ and $\DD_{2(q+1)}$ is always non-trivial, which plays an important role in the classification of block designs.

In Section \ref{SectionLocalAction}, we obtain some general results on the automorphism groups of symmetric designs satisfying certain local properties.
In Sections \ref{SectionAbelianFrobenius} and \ref{SecProof1}, we study locally transitive symmetric designs with a local abelian or local dihedral action, and find that such an automorphism group is either imprimitive on both points and blocks, or a Frobenius group of odd order.
In Section \ref{pri2designs}, we classify non-symmetric designs admitting a locally transitive and point-primitive automorphism group with dihedral point stabilizers, which includes all point-primitive, point-locally dihedral designs for which the point-local action is faithful.

Throughout the paper, we adopt the following notation:
 $B$ denotes a block of a design $\D$, and $x$ denotes a point of $\D$.
We denote by $\DD_{2n}$ the dihedral group of order $2n$ ($n\geq3$), by $N$ the cyclic normal subgroup of order $n$ in $\DD_{2n}$, and by $F$ a complement of $N$ in $\DD_{2n}$.
Our first main result concerns locally dihedral and locally abelian symmetric designs, stated as Theorem \ref{ThmSym} below.

\begin{theorem}\label{ThmSym}
Let $\D=(\mathcal{P}, \mathcal{B})$ be a symmetric design with $\lambda>1$ and $G$ be a locally transitive automorphism group of $\D$.
Let $x\in\P$ and $B\in\B$.
If $K$ is an abelian group or a dihedral group, then the following are equivalent:
\begin{itemize}
\item[{\rm(a)}]  $G_x\cong K$.
\item[{\rm(b)}] $G_B\cong K$.
\item[{\rm(c)}] $G_x^{\D(x)}\cong K$.
\item[{\rm(d)}] $G_B^B\cong K$.
\item[{\rm(e)}] The block set $\B$ is generated by a faithful suborbit of $G$ on the point set $\P$, namely, $\B=\{\Delta^g\,|\,g\in G\}$ with $\Delta$ a faithful orbit of $G_x$ on $\P$, such that $G_x\cong G_x^{\Delta}\cong K$. 
\end{itemize}
Assume that one of {\rm (a)--(e)} holds.
Then $G_x$ and $G_B$ are conjugate in $G$. Moreover, either
\begin{itemize}
\item[\rm(1)] $G$ is imprimitive on both $\P$ and $\B$, or

\item[\rm(2)] $K$ is cyclic, and $G$ is a Frobenius group of odd order with $G\leq\AGammaL_1(v)$.
\end{itemize}
\end{theorem}

\begin{remark}\rm
Two known imprimitive examples of locally dihedral symmetric designs are the $2$-$(16,6,2)$ design and the $2$-$(15,8,4)$ design.
The $2$-$(16,6,2)$ design admits two flag-transitive, imprimitive automorphism groups, namely $\ZZ_2^2{:}\SS_4$ and $(\ZZ_4^2{:}\ZZ_3){:}\ZZ_2$, both with stabilizers $\DD_6$.
This design also admits two flag-transitive, imprimitive automorphism groups, $(\ZZ_2^2{:}\SS_4){:}\ZZ_2$ and $(\ZZ_4^2{:}\ZZ_3){:}\ZZ_2^2$,  with stabilizers $\DD_{12}$. The full automorphism group of this design is $\ZZ_2^4{:}\SS_6$, as described in \cite{RegueiroReduction}.
The $2$-$(15,8,4)$ design is isomorphic to the design derived from the points and complements of hyperplanes of the projective space
$\PG(3, 2)$.
Praeger and Zhou \cite[Proposition 1.5]{PraegerZhou} proved that this design admits a flag-transitive, imprimitive automorphism group $\SS_5$ with stabilizers $\DD_8$.

Examples of locally abelian symmetric designs also exist.
Let $G$ be the affine group $\ZZ_p{:}\ZZ_{\frac{p-1}{2}}\leq \AGL_1(p)$ acting on a set $\mathcal{P}$ of size $p$.
If $p\equiv 3\pmod{4}$, then $\frac{p-1}{2}$ is odd, and $G$ is a Frobenius group with exactly two suborbits of length $\frac{p-1}{2}$. By \cite[Proposition 3.1]{Kantor}, $G$ is 2-homogeneous and not 2-transitive on $\P$.
Let $\Delta$ be a suborbit of $G$. Then $\mathcal{D}=(\P,\{\Delta^g|g\in G\})$ is a Paley design (\cite[Theorem 1.1]{Kantor}) admitting a primitive, locally abelian automorphism group $G$.
Besides, we note that the above $2$-$(16,6,2)$ design also admits three imprimitive, locally abelian automorphism groups, namely, $(\ZZ_2^4{:}\ZZ_2){:}\ZZ_3$, $\ZZ_2^3.\AA_4$ and $(\ZZ_4{:}\DD_8){:}\ZZ_3$, with stabilizers $\ZZ_6$.
\end{remark}

Our second main result provides a classification of locally transitive, point-primitive designs with dihedral point stabilizers, which covers all point-locally dihedral, point-primitive designs for which the point-local action is faithful.

\begin{theorem}\label{ThmMain}
Let $\D$ be a $2$-$(v,k,\lambda)$ design and $G$ be a locally transitive, point-primitive automorphism group of $\D$ with point stabilizer $G_x$ a dihedral group. Then $\D$ is non-symmetric, and one of the following holds:
\begin{enumerate}
\item[\rm(a)]  $\D$ is a Witt-Bose-Shrikhande Space with $\lambda=1$, and $G=\PSL_2(q)\,(q=2^f\geq8)$ with $G_x=\DD_{2(q+1)}$.

\item[\rm(b)]  $\D$ is a $2$-$(6,3,2)$ design with $G=\PSL_2(5)$ and $G_x=\DD_{10}$.

\item[\rm(c)]  $\D$ is a $2$-$(10,4,2)$ design with $G=\PGL_2(5)$ and $G_x=\DD_{12}$.

\item[\rm(d)] $\D$ is a $2$-$(28,7,2)$ design with $G=\PSL_2(8)$ and $G_x=\DD_{18}$.

\item[\rm(e)]  $\D$ is a $2$-$(36,6,2)$ design with $G=\PSL_2(8)$ and $G_x=\DD_{14}$.

\item[\rm(f)]  $G=\ZZ_p^d{:} \DD_{2n}$ is a solvable affine group.
\end{enumerate}
\end{theorem}

\begin{remark}\rm$\,$
\begin{itemize}
\item[(1)] All groups $G$ in Theorem \ref{ThmMain}\,(a)--(e) have a faithful local action of $G_x$ on $\D(x)$.
\item[(2)] The affine space ${\rm AG}(2,4)$ is an example for Theorem \ref{ThmMain}\,(f), namely, a $2$-$(16,4,1)$ design, admitting a flag-transitive automorphism group $G=\ZZ_2^{4}{:}\DD_{10}\leq {\rm AGL}_4(2)$.
\end{itemize}
\end{remark}

In Section \ref{SecProof2}, an application of the main theorems is given in Corollary \ref{ThmGB2p}, where we prove that a symmetric design admitting a flag-transitive automorphism group with point stabilizers of order $2p$ ($p$ is prime) is a unique $2$-$(16,6,2)$ design.
This is motivated in part by Quirin \cite{QuirinSmallOrbitals} and Wang \cite{WangJiesuborbitun2p}, who previously studied primitive groups with point stabilizers $\DD_{2p}$.

\medskip

\section{Primitive groups with dihedral point stabilizers}\label{SecPriDihedral}

To prove the main theorems, we need an explicit classification of primitive permutation groups with dihedral point stabilizers.
This classification problem was first studied by Quirin \cite{QuirinSmallOrbitals}, who proved that if a primitive group $G$ with point stabilizers isomorphic to $\DD_{2p}$ with $p$ prime has a faithful suborbit, then $G$ is simple or solvable.
Afterward, simple primitive groups with point stabilizers $\DD_{2p}$ were classified by Wang \cite{WangJiesuborbitun2p}.
In 2009, Li \cite[Lemma 3.1]{LicaihengD2n} tackled primitive groups with general dihedral point stabilizers and determined the socles.
Further, a classification of primitive groups with solvable stabilizers was studied by Li and Zhang \cite{ZhanghuaLicaiheng} in 2011.
While,  some reduction details were omitted in \cite{LicaihengD2n}, and the complete classification is not presented in both \cite{LicaihengD2n} and \cite{ZhanghuaLicaiheng}.
We supplement the reduction details in Lemma \ref{LemmaDihedralAffAlmost} and give a complete classification in Theorem \ref{ProLCH}.

We note that if a primitive group $G\leq\Sym(\Omega)$ with non-trivial abelian point stabilizers is a Frobenius group.
In fact, for any distinct points $x_1, x_2 \in \Omega$, if $G_{x_1}=G_{x_2}$, then $G$ is regular. So $G_{x_1}\ne G_{x_2}$.
Since $G_x$ is abelian, for $g \in G_{x_1} \cap G_{x_2}$, we have $\langle g\rangle\unlhd\langle G_{x_1}, G_{x_2}\rangle=G$. But $\langle g\rangle$ has fixed points on $\P$, yielding that $\langle g\rangle$ fixes every point of $\P$, namely, $g=1$. Thus $G$ is a Frobenius group.
Further, by \cite{LiebeckCyclicStabilizer}, the point stabilizer is cyclic.
A primitive Frobenius group is an affine group, and the structure of Frobenius groups can be referred to \cite[Section 3.4]{Dixon}, for example.

\begin{theorem}\label{ProLCH}
If $G$ is a primitive group with dihedral point stabilizers,
then ${\rm soc}(G)$ is either elementary abelian or non-abelian simple.
Further, if ${\rm soc}(G)$ is non-abelian simple, then $G$, ${\rm soc}(G)$ and $G_x$ are listed in Table {\rm\ref{TablePriGroupDihedralSta}}, up to conjugacy.
\end{theorem}

\begin{table}[h]
	\caption{Almost simple primitive groups with dihedral stabilizers}
	\label{TablePriGroupDihedralSta}
	\centering
	\begin{tabular}{llll}
		\hline
		$G$\quad\quad & ${\rm soc}(G)$\quad\quad\quad\quad & $G_x$ & Remark\\
		\hline
		$\AA_5$\quad\quad & $\AA_5$\quad\quad\quad\quad &  $\DD_{6}$, $\DD_{10}$\\
		
		$\SS_5$\quad\quad & $\AA_5$\quad\quad\quad\quad & $\DD_{12}$\\

		$\PSL_2(q)$\quad\quad & $\PSL_2(q)$\quad\quad\quad\quad & $\DD_{2(q+1)/(2,q-1)}$& $q\neq7,9$\\
		
		\quad\quad &  \quad\quad\quad\quad& $\DD_{2(q-1)/(2,q-1)}$ & $q\neq5,7,9,11$\\
		
		&&$\DD_{10}$ &$q=5$\\

		$\PGL_2(q)$\quad\quad & $\PSL_2(q)$ \quad\quad\quad\quad& $\DD_{2(q+1)}$ &$q$ is odd\\

		\quad\quad &\quad\quad\quad\quad & $\DD_{2(q-1)}$ &$q$ is odd, $q\neq5$\\
		
		$\PSL_3(2).2$\quad\quad & $\PSL_3(2)$\quad\quad\quad\quad & $\DD_{12}$, $\DD_{16}$\\
		
		$\Sz(q)$\quad\quad & $\Sz(q)$\quad\quad\quad\quad & $\DD_{2(q-1)}$\\
		
		${\rm^2G_2}(3)'$\quad\quad & ${\rm^2G_2}(3)'$\quad\quad\quad\quad & $\DD_{14}$, $\DD_{18}$\\
		\hline
	\end{tabular}
\end{table}

\begin{remark}\label{RemarkOfTable1}\rm
Table \ref{TablePriGroupDihedralSta} lists all cases according to the types of their socles, which contains duplications.
The isomorphic relations are:
$\AA_5\cong\PSL_2(4)\cong\PSL_2(5)$, $\SS_5\cong\PGL_2(5)$, ${\rm^2G_2}(3)'\cong\PSL_2(8)$, $\PSL_3(2).2\cong\PGL_2(7)$ and $\PSL_3(2)\cong\PSL_2(7)$.
Moreover, except $G=\PSL_2(5)$ with its dihedral stabilizer $G_x=\DD_{10}$ lying in the $\mathcal{C}_1$-subgroup of the Aschbacher-classes, every dihedral stabilizer is in the classes $\mathcal{C}_2$ or $\mathcal{C}_3$.
\end{remark}

In the following we prove Theorem \ref{ProLCH}.
The classification relies on the O'Nan-Scott Theorem, which classifies finite primitive permutation groups.
The five-types version is given below.

\begin{lemma}{\rm(O'Nan-Scott theorem {\rm\cite{ONan1988}})}\label{LemmaONan}
If $G$ is a finite primitive permutation group, then $G$ is one of the following types:
$(a)$ affine; $(b)$ almost simple; $(c)$ simple diagonal; $(d)$ product action; and $(e)$ twisted wreath product.
\end{lemma}

Inspired by the argument in \cite[Section 3.1]{LicaihengD2n}, we have the following lemma.

\begin{lemma}\label{Tx_not_abelian}
  Let $G$ be an almost simple primitive group on $\Omega$ with socle $T$, and  let $x \in \Omega$. Then $T_x$ is not abelian.
\end{lemma}
\begin{proof}
  Assume that $T_x$ is abelian. Let $P$ be a Sylow $p$-subgroup of $T_x$. Since $P$ is a characteristic subgroup of $T_x$ and $T_x \lhd G_x$, we have $P \lhd G_x$ and $G_x \leqslant N_G(P) \leqslant G$. Note that $T$ is the minimal normal subgroup of $G$ and $P \leqslant T_x <T$. Hence $P$ is not a normal subgroup of $G$ and $N_G(P)\ne G$. Since $G_x$ is maximal on $G$, we have $G_x = N_G(P)$ and $T_x = N_T(P)$. As $T_x$ is abelian, $T_x = N_T(P)=C_T(P)$.

Claim that $P$ is a Sylow $p$-subgroup of $T$.
Let $S$ be a Sylow $p$-subgroup of $T$ containing $P$. If $P<S$, then by the properties of $p$-groups, $P<N_S(P)$. Thus there exist $n \in N_S(P) \setminus P$, such that $n \in N_S(P) \leqslant N_T(P)=T_x$. Hence $n$ lies in the unique Sylow $p$-subgroup $P$ of $T_x$, a contradiction. Then we have $P=S$ is a Sylow $p$-subgroup of $T$.

Now, by $N_T(P)=C_T(P)$ and a well-known theorem of Burnside, $T$ is $p$-nilpotent, that is $T=T_{p'}{:}T_p$, where $T_p$ is a Sylow $p$-subgroup of $T$ and $T_{p'}$ is a Hall $p'$-subgroup of $T$. It is impossible. Hence $T_x$ is not abelian.
\end{proof}

\begin{lemma}\label{LemmaDihedralAffAlmost}
If $G$ is a primitive group with dihedral point stabilizers, then $G$ is of affine type or almost simple type.
\end{lemma}

\begin{proof}
Note that the dihedral point stabilizer is solvable.
For primitive groups of twisted wreath product type, the point stabilizer has a non-abelian composition factor (see \cite[Section 1]{ONan1988}, for example).
For primitive groups of simple diagonal type, the point stabilizer ${\rm soc}(G)_x$ of the socle is isomorphic to a non-abelian simple group.
Hence these two types are ruled out.

Suppose that $G$ is of product action type, acting on the set $\Delta^{\ell}$ with $\ell\geq2$.
Then ${\rm soc}(H)^{\ell}={\rm soc}(G)\leq G\leq H\wr S_{\ell}$, where $H$ is a primitive group of almost simple type or diagonal type on the set $\Delta$.
Let $x=(\alpha,\alpha,\ldots,\alpha)\in\Delta^{\ell}$ with $\alpha\in\Delta$.
Then ${\rm soc}(H)_{\alpha}^{\ell}={\rm soc}(G)_x\unlhd G_x\cong D_{2n}$.
If $H$ is of simple diagonal type, then ${\rm soc}(H)_{\alpha}$ is a non-abelian simple group, which is clearly impossible.
If $H$ is of almost simple type, then ${\rm soc}(G)=T^{\ell}$ and ${\rm soc}(G)_x= (T_{\alpha})^{\ell}$, where $T={\rm soc}(H)$ is non-abelian and simple.
Now $T_{\alpha} \lhd (T_{\alpha})^{\ell} \lhd G_x=\DD_{2n}$.
For a dihedral group $\DD_{2n}$, each proper normal subgroup is cyclic, $\ZZ_2\times \ZZ_2$ (if $n=4$), or $\DD_{n}$ (if $n \geqslant 6$ is even).
By Lemma \ref{Tx_not_abelian}, we know that $T_{\alpha}$ is not abelian.
Hence $ (T_{\alpha})^{\ell}$ is not abelian, and so both $ (T_{\alpha})^{\ell}$ and $T_{\alpha}$ are dihedral.
But this implies that  $[(T_{\alpha})^{\ell}:T_{\alpha}]=2$, a contradiction.
\end{proof}

\begin{lemma}\label{LemmaDihedralPSL2q}
If $G$ is an almost simple primitive group with socle $\PSL_2(q)$, where $q\geq5$ is a prime power, admitting a dihedral point stabilizer $G_x$, then $G$ and $G_x$ are listed in Table {\rm\ref{TablePSL}}.
\end{lemma}

\begin{table}[h]
	\caption{Dihedral maximal subgroups of almost simple groups with socle $\PSL_2(q)$}
	\label{TablePSL}
	\centering
	\begin{tabular}{lll}
		\hline
		$G$\quad\quad \quad\quad\quad\quad & $G_x$ \quad\quad \quad\quad\quad\quad\quad\quad& Remark\\
		\hline
		$\PSL_2(q)$\quad\quad \quad\quad\quad\quad &$\DD_{2(q+1)/(2,q-1)}$& $q\neq7,9$\\
		&$\DD_{2(q-1)/(2,q-1)}$ & $q\neq5,7,9,11$\\
		&$\DD_{10}$ &$q=5$\\
		\hline
		$\PGL_2(q)$\quad\quad \quad\quad\quad\quad& $\DD_{2(q+1)}$ & $q$ is odd\\
		\quad\quad \quad\quad\quad\quad & $\DD_{2(q-1)}$ & $q$ is odd, $q\neq5$\\
		\hline
	\end{tabular}
\end{table}

\begin{proof}
Let $T$ be the socle of an almost simple primitive group $G$.
Then $T_x\unlhd G_x$, and so $T_x$ is either dihedral, or cyclic, or $\ZZ_2\times \ZZ_2$ with $G_x\cong \DD_8$.
By Lemma \ref{Tx_not_abelian}, $T_x$ is not abelian.
Hence $T_x$ must be dihedral.

Let $T=\PSL_2(q)$, $q=p^f$, and $d=\gcd(2,q-1)$.
When $T_x=T\cap G_x$ is not maximal in $T$, the cases are listed in \cite[Theorem 1.1 and Table 1]{MGiudici}.
These are $G=\PGL_2(7)$ with $G_x=\DD_{12}$ or $\DD_{16}$, $G=\PGL_2(9)$ with $G_x=\DD_{16}$ or $\DD_{20}$ and $G=\PGL_2(11)$ with $G_x=\DD_{20}$.
We then tackle the case when $T_x$ is maximal in $T$.
By \cite[Theorem 2.1 and 2.2]{MGiudici}, we have $T_x=\DD_{2(q\pm1)}$ ($q$ is even), or $T_x=\DD_{q-1}$ ($q$ is odd and $q\notin\{5,7,9,11\})$, or $T_x=\DD_{q+1}$($q$ is odd and $q\notin\{7,9\}$), or $T_x=\DD_{10}$($q=5$), as stated in the line 2 of Table \ref{TablePSL}.
Now $G_x=T_x.O$, where $O\cong G/T\leq\Out(T)$. Let $\phi$ be the field automorphism of $\GF_q$ taking $p$-th power, $\mu$ be a generator of $\GF_q^\times$, and
\[
\delta=
\begin{bmatrix}
\mu & 0\\
0 & 1
\end{bmatrix}
\in\GL_2(q).
\]
Let $\,\overline{\phantom{w}}\,$ be the homomorphism from $\GammaL_2(q)$ to $\PGammaL_2(q)$ modulo scalars, and let $\overline{\delta}$ and $\overline{\phi}$ be the image of $\PGammaL_2(q)=\Aut(T)$ in $\Out(T)=\Aut(T)/T$.
Then $\Out(T)=\langle\overline{\delta}\rangle\times\langle\overline{\phi}\rangle$ with $o(\overline{\delta})=d$ and $o(\overline{\phi})=f$.

Since $T_x$ and $G_x$ are both dihedral, it follows that $O\cong G_x/T_x$ is a cyclic subgroup of $\Out(T)$, and so $|O|=1$ or $2$.
If $|O|=1$, then $G=\PSL_2(q)$ and this corresponds to the first line of Table \ref{TablePSL}.
In the following we assume that $|O|=2$.
Then $O=\langle\overline{\delta}\rangle$ or $\langle\overline{\delta}^i\overline{\phi}^e\rangle$, where $i\in\{0,1\}$, $o(\overline{\phi}^e)=2$ and $f=2e$.
If $O=\langle\overline{\delta}\rangle$, then $G=\PGL_2(q)$ and the maximal subgroups are listed in \cite[Theorem 3.5]{MGiudici}.
These dihedral maximal subgroups are $\DD_{2(q-1)}$ ($q\neq5$) and $D_{2(q+1)}$.
In the following we assume that $O=\langle\overline{\delta}^i\overline{\phi}^e\rangle$.

First assume that $T_x=\DD_{\frac{2(q-1)}{d}}$ is a $\mathcal{C}_2$-subgroup in the Aschbacher-classes \cite{Aschbacher}.
Since the maximal subgroups $G_x$ of such a type are conjugate, assume without loss of generality that $T_x=\langle a\rangle{:}\langle b\rangle$ is normalized by $\delta$ and $\phi$, where the preimages of $a$ and $b$ in $\GL_2(q)$ are
\[
\begin{bmatrix}
\lambda & 0\\
0 & \lambda^{-1}
\end{bmatrix}
\,\,\mbox{and}\,\,
\begin{bmatrix}
0 & 1\\
-1 & 0
\end{bmatrix},
\]
respectively, with $\lambda$ generating $\GF_q^{\times}$. It follows that
\[
a^{\phi^e}
=\overline{\begin{bmatrix}
\lambda & 0\\
0 & \lambda^{-1}
\end{bmatrix}}^{\phi^e}
=\overline{\begin{bmatrix}
\lambda^{p^e} & 0\\
0 & \lambda^{-p^e}
\end{bmatrix}}
=a^{p^e}
\]
and
\[
a^{\delta}
=\overline{\begin{bmatrix}
\lambda & 0\\
0 & \lambda^{-1}
\end{bmatrix}}^{\delta}
=\overline{\begin{bmatrix}
\mu^{-1}\lambda\mu & 0\\
0 & \lambda^{-1}
\end{bmatrix}}
=\overline{\begin{bmatrix}
\lambda & 0\\
0 & \lambda^{-1}
\end{bmatrix}}
=a.
\]
Note that $O=\langle\overline{\delta}^i\overline{\phi}^e\rangle=\langle\overline{\delta^i\phi^e}\rangle$, whence $G=\langle T, \delta^i\phi^e\rangle$ and $G_x=\langle T_x, \delta^i\phi^e\rangle$. Therefore, $\delta^i\phi^e$ is either contained in the cyclic group $N$ of index 2 in $G_x$, or an involution contained in $G_x\setminus N$.
In the former case, $\delta^i\phi^e\in C_{G_x}(N)$ and so
\[
a=a^{\delta^i\phi^e}=a^{p^e},
\]
which yields $a^{p^e-1}=1$, contradicting $o(a)=\frac{q-1}{d}=\frac{p^{2e}-1}{d}$ .
For the latter case, $\delta^i\phi^e$ maps $a$ to its inverse by conjugation, namely,
\[
a^{-1}=a^{\delta^i\phi^e}=a^{p^e}.
\]
This forces that $\frac{p^{2e}-1}{d}=o(a)$ divides $p^e+1$, and so $q=p^{2e}=9$. However, this is not possible by~\cite{atlas}.

Next assume that $T_x\cong \DD_{\frac{2(q+1)}{d}}$ is a $\mathcal{C}_3$-subgroup in the Aschbacher-classes.
Then the normalizer $H$ of $T_x$ in
$\mathrm{P\Gamma L}_2(q)$ satisfies $H=\langle\overline{\omega}\rangle{:}\langle\overline{\psi}\rangle$, where $\omega$ is a generator of $\GF_{q^2}^{\times}$ and $\psi$ is a field automorphism over $\GF_{q^2}$, such that
$\overline{\omega}^{\overline{\psi}}=\overline{\omega}^p$,
\[
H\cap\PGL_2(q)=\langle \overline{\omega}\rangle{:}\langle\overline{\psi}^{2e}\rangle\ \text{, and }\ T_x=H\cap\PSL_2(q)=\langle\overline{\omega}^d\rangle{:}\langle\overline{\psi}^{2e}\rangle.
\]

Since $G_x/T_x\cong O=\langle\overline{\delta}^i\overline{\phi}^e\rangle$, $T_x$ has index $2$ in $G_x \leqslant H$, and there exists $h=\overline{\omega}^s\overline{\psi}^t\in G_x$ for some integers $s$ and $t$, such that $G_x=\langle T_x, h \rangle$. Note that $h^2 =\overline{\omega}^s\overline{\psi}^t\overline{\omega}^s\overline{\psi}^t =\overline{\omega}^s(\overline{\psi}^t \overline{\omega}^s(\overline{\psi}^{t})^{-1}) \overline{\psi}^{2t}\in T_x$. Then $t=0$ or $e$.
If $t=0$, then $G=\PGL_2(q)$, as stated in the line 3 of Table \ref{TablePSL}. Hence, $t=e$.
As $G_x$ is dihedral and $\langle\overline{\omega}^d\rangle$ is contained in its cyclic subgroup of index $2$, it follows that $h$ either commutes with $\overline{\omega}^d$ or maps $\overline{\omega}^d$ to its inverse by conjugation. Thus,
\[
\overline{\omega}^{\pm d}=(\overline{\omega}^d)^h=(\overline{\omega}^d)^{\overline{\omega}^s\overline{\psi}^e}
=(\overline{\omega}^d)^{\overline{\psi}^e}=(\overline{\omega}^{dp^e}),
\]
and so $p^{2e}+1=o(\overline{\omega})$ divides $d(p^e\pm1)$, not possible.

Finally, the only group having a dihedral maximal subgroup of class $\mathcal{C}_1$ is $\PSL_2(5)$ with $G_x=\DD_{10}$.
Hence Table~\ref{TablePSL} is complete.
\end{proof}

\begin{lemma}\label{LemmaDihedralSuzuki}
If $G$ is an almost simple primitive group with socle ${\rm soc}(G)=\Sz(q)$, admitting a dihedral point stabilizer $G_x$, then $G=\Sz(q)$ with $G_x\cong \DD_{2(q-1)}$.
\end{lemma}

\begin{proof}
If ${\rm soc}(G)<G$, then $G_x$ is an extension of ${\rm soc}(G)_x=\Sz(q)_x$ by an outer automorphism of order 2.
While $|\Out(\Sz(q))|$ is odd, this is impossible. So $G=\Sz(q)$ and $G_x\cong \DD_{2(q-1)}$.
\end{proof}

\medskip

\noindent{\bf Proof of Theorem \ref{ProLCH}.}
According to \cite[Lemma 3.1]{LicaihengD2n}, the only possible socle of an almost simple group with dihedral stabilizers is $\PSL_2(q)$ and $\Sz(q)$.
So it is obtained immediately from Lemmas \ref{LemmaONan}-\ref{LemmaDihedralSuzuki} and by adding the isomorphic cases.
$\square$

\medskip

The following proposition studies the intersection of two maximal subgroups in $\PGL_2(q)$, respectively of type $\DD_{2(q-1)}$ and $\DD_{2(q+1)}$. This result is useful in the proof of classifying point-locally dihedral $2$-designs.

\begin{proposition}\label{LemmaMeet>1}
Let $G=\PGL_2(q)$ with $q$ odd. Then any maximal subgroup of $G$ of type $\DD_{2(q-1)}$ has non-trivial intersection with any maximal subgroup of type $\DD_{2(q+1)}$.
\end{proposition}

\begin{proof}

We identify $G=\PGL_2(q)$ as the unitary group $\PGU_2(q)$ over the unitary space $\GF_{q^2}^2$ with a non-singular conjugate-symmetric sesquilinear form $\beta$.

Let $\{\boldsymbol{e}, \boldsymbol{f}\}$ be a basis for the unitary space $\GF_{q^2}^2$ such that $\beta(\boldsymbol{e},\boldsymbol{e})=\beta(\boldsymbol{f},\boldsymbol{f})=0$ and $\beta(\boldsymbol{e},\boldsymbol{f})=1$, and let $H=G_{\{\langle\boldsymbol{e}\rangle,\langle\boldsymbol{f}\rangle\}}\cong \DD_{2(q-1)}$.
Let $\boldsymbol{u}$ and $\boldsymbol{v}$ be any non-singular vectors such that $\boldsymbol{u}\bot\boldsymbol{v}$. Since $\beta(\boldsymbol{u},\boldsymbol{u}) \ne 0$, we have $\boldsymbol{u}=\lambda_1\boldsymbol{e}+\lambda_2\boldsymbol{f}$ with $\lambda_1, \lambda_2 \in\GF_{q^2} \setminus \{0\}=\GF_{q^2}^{\times}$. Replacing $\boldsymbol{u}$ by its nonzero multiple, we have $\boldsymbol{u}=\boldsymbol{e}+\lambda\boldsymbol{f}$ for some $\lambda\in\GF_{q^2}^{\times}$. Then let $K=G_{\{\langle\boldsymbol{u}\rangle,\langle\boldsymbol{v}\rangle\}}\cong \DD_{2(q+1)}$.
Since $\boldsymbol{u}\bot\boldsymbol{v}$, replacing $\boldsymbol{v}$ by its nonzero multiple if necessary, we have $\boldsymbol{v}=\boldsymbol{e}-\lambda^q\boldsymbol{f}$.
Take $\mu\in \GF_{q^2}^{\times}$ such that $\mu^{q+1}=-1$.
Let $a=\lambda\mu$ and $A\in \GL_2(q^2)$ such that $\boldsymbol{e}^A=a\boldsymbol{f}$ and $\boldsymbol{f}^A=a^{-q}\boldsymbol{e}$.
Then $A$ preserves $\{\langle\boldsymbol{e}\rangle,\langle\boldsymbol{f}\rangle\}$ and $\beta$. In particular, $A\in \GU_2(q)$. Moreover,
\begin{align*}
\boldsymbol{u}^A=\boldsymbol{e}^A+\lambda\boldsymbol{f}^A&=a\boldsymbol{f}+\lambda a^{-q}\boldsymbol{e}\\
&=\lambda a^{-q}(\boldsymbol{e}+\lambda^{-1}a^{q+1}\boldsymbol{f})
=\lambda a^{-q}(\boldsymbol{e}-\lambda^{q}\boldsymbol{f})=\lambda a^{-q}\boldsymbol{v},\\
\boldsymbol{v}^A=\boldsymbol{e}^A-\lambda^q\boldsymbol{f}^A&=a\boldsymbol{f}-\lambda^q a^{-q}\boldsymbol{e}\\
&=-\lambda^q a^{-q}(\boldsymbol{e}-\lambda^{-q}a^{q+1}\boldsymbol{f})
=-\lambda^q a^{-q}(\boldsymbol{e}+\lambda\boldsymbol{f})=\lambda^q a^{-q}\boldsymbol{u},
\end{align*}
and so $A$ preserves $\{\langle\boldsymbol{u}\rangle,\langle\boldsymbol{v}\rangle\}$. Consequently, the image of $A$ in $\PGU_2(q)$ lies in $H\cap K$. Since $A$ is not a scalar, it follows that $H\cap K>1$.
Note that all dihedral subgroups of order $2(q+1)$ in $G$ are conjugate. For any maximal subgroup $L$ of type $\DD_{2(q+1)}$, we have $L=K^g=G_{\{\langle\boldsymbol{u^g}\rangle,\langle\boldsymbol{v^g} \rangle\}}$ where $\boldsymbol{u^g}$ and $\boldsymbol{v^g}$ are also non-singular perpendicular vectors. Hence $H \cap L$ is also non-trivial. Since all dihedral subgroups of order $2(q-1)$ in $G$ are conjugate, this proposition holds.
\end{proof}

\medskip

\section{Some preliminary results on block designs}\label{SecPre}

In this section we present some preliminary results concerning block designs and their automorphism groups, which are necessary in the proof of the main results.
The following are basic arithmetic properties of a 2-design. Some of these results are well known and can be found in \cite[Sections 3.2.3, 3.2.4, 3.2.9]{Biggs}, for examples.
\begin{lemma}\label{LemmaParameters}
The parameters of a $2$-design $\mathcal{D}$ satisfy the following:
\begin{enumerate}
\item[\rm(a)]  $r(k-1)=\lambda(v-1)$.

\item[\rm(b)]  $bk=vr$.

\item[\rm(c)] $b\geq v$ {\rm(}Fisher's inequality{\rm)}.

\item[\rm(d)] $\lambda v<r^2$.
\end{enumerate}
If $\D$ is a symmetric design, then these are reduced to
\begin{enumerate}
\item[\rm(a$'$)]  $k(k-1)=\lambda(v-1)$.

\item[\rm(c$'$)]  $b=v$ and $r=k$.

\item[\rm(d$'$)] $\lambda v<k^2$.
\end{enumerate}
\end{lemma}

\begin{lemma}{\rm\cite[Proposition 1.1]{Cameron}}\label{LemmaLargegroup}
Let $X$ be a permutation group on $\Omega$, and $B\subseteq\Omega$.
If $\mathcal{D}=(\Omega,B^X)$ is a $2$-design, then for any $M$ with $X\leq M\leq \Sym(\Omega)$, $\mathcal{D}=(\Omega,B^M)$ is also a $2$-design.
\end{lemma}

The following are well known properties of flag-transitive automorphism groups of a 2-design.
\begin{lemma}\label{LemmaFlagtrProperties}
Let $G$ be an automorphism group of a $2$-design. Then the following are equivalent:
\begin{enumerate}
\item[\rm(a)] $G$ is flag-transitive.

\item[\rm(b)] $G$ is transitive on $\mathcal{B}$ and $G_B$ is transitive on $B$ for any block $B\in\B$.

\item[\rm(c)] $G$ is transitive on $\mathcal{P}$ and $G_x$ is transitive on  $\D(x)$ for any point $x\in\P$.
\end{enumerate}
\end{lemma}

\begin{lemma}\label{LemmaCubic}
Let $G$ be a flag-transitive automorphism group of a $2$-design with point stabilizer $G_x$.
Then $|G|<|G_x|^3$.
\end{lemma}

Note that a subgroup $H$ of a finite group $G$ is said to be large if $|G|\leq|H|^3$.
Hence the point stabilizer of a flag-transitive automorphism group of a $2$-design is a large subgroup.
Alavi and Burness \cite{AlaviBurnessLarge} determined all the large maximal subgroups of finite simple groups.

\begin{lemma}{\rm\cite[p.1]{DaviesImpri}}\label{LemmaDavies}
Let $G$ be a flag-transitive automorphism group of a $2$-design, and $\Gamma$ be a non-trivial orbit of some point stabilizer $G_x$.
If $B$ is a block through $x$, then $r|\Gamma\cap B|=\lambda|\Gamma|$.
\end{lemma}

An automorphism group $G$ is said to be flag-regular if $G$ acts regularly on the set of flags. Such groups are also called sharply flag-transitive.

\begin{lemma}\label{LemmaDefFlagregular}
If $G$ is a block-transitive automorphism group of a $2$-design, then the following are equivalent:
\begin{enumerate}
\item[\rm (a)] $G$ is flag-regular.

\item[\rm (b)]  $G_x$ acts regularly on $\D(x)$.

\item[\rm (c)]  $G_B$ acts regularly on $B$.

\item[\rm (d)] $G$ is flag-transitive and $|G|=vr$.
\end{enumerate}
\end{lemma}

Some useful results concerning fixed points of automorphisms of symmetric designs are presented below:

\begin{lemma}\label{LemmaSameFixed}
An automorphism of a symmetric design fixes an equal number of fixed points and blocks.
\end{lemma}

\begin{lemma}\label{LemmaInvolutionFixSym}{\rm\cite[Proposition 4.23]{Lander}}
{\rm
Let $\D$ be a non-trivial symmetric design and $g\in{\rm Aut}(\D)$ be an involution.
If $|{\rm Fix}_{\mathcal{P}}(g)|>0$, then one of the following holds:
\begin{enumerate}
\item[\rm(a)]  If $k$ and $\lambda$ are even, then $|{\rm Fix}_{\mathcal{P}}(g)|\geq\frac{k}{\lambda}+1$;
\item[\rm(b)]  For other cases, $|{\rm Fix}_{\mathcal{P}}(g)|\geq\frac{k-1}{\lambda}+1$.
\end{enumerate}
}
\end{lemma}

If $g$ is an automorphism of a symmetric design $\D$ and $g$ fixes each point of a block $B$, then we say $g$ is an axial automorphism with axis $B$.
Dually, if $g$ fixes each block through a given point $x$, we say $g$ is a central automorphism with centre $x$. The centre $x$ is a fixed point of $g$, and the axis $B$ is a fixed block of $g$.
Lemmas \ref{LemmaOneAxis} and \ref{LemmaAxisOutsideCentre} are important results on axial (and central) automorphisms.

\begin{lemma}\label{LemmaOneAxis}{\rm\cite[\S 2.3.15]{Dembowski}}
An automorphism of a symmetric design has at most one axis and one centre.
\end{lemma}

\begin{lemma}\label{LemmaAxisOutsideCentre}{\rm\cite[\S 2.3.17]{Dembowski}}
Let $g$ be an axial automorphism of a symmetric design with axis $B$. If $g$ fixes a point $x\in\P\setminus B$, then $g$ is a central automorphism with centre $x$.
Dual statement holds if $g$ is central.
\end{lemma}

The following is a classification of flag-transitive projective planes {\rm\cite[Theorem A]{KantorProplane}}.

\begin{lemma}\label{LemmaKantorProPlane}
If $G$ is a flag-transitive automorphism group of a projective plane $\D$, then one of the following holds:
\begin{enumerate}
\item[\rm(a)] $\D$ is Desarguesian and $\PSL_3(q)\unlhd G\leq \PGammaL_3(q)$.

\item[\rm(b)] $G$ is a Frobenius group of odd order $(n^2+n+1)(n+1)$, and
$n^2+n+1$ is prime.
\end{enumerate}
\end{lemma}






The following is the classification of flag-transitive symmetric designs with $\lambda=2$ (also called biplanes){\rm(\cite{RegueiroReduction}-\cite{RegueiroSurvey})}, which are useful in the present paper.

\begin{lemma}\label{LemmaLambda=2}
If $\D$ is a biplane admitting a flag-transitive automorphism group $G$, then one of the following holds:
\begin{enumerate}
\item[\rm(a)] $\D$ is the trivial $2$-$(4,3,2)$ design.

\item[\rm(b)] $\D$ is the unique $2$-$(7,4,2)$ design with $\Aut(\D)=\PSL_2(7)$.

\item[\rm(c)] $\D$ is the unique $2$-$(11,5,2)$ design with $\Aut(\D)=\PSL_2(11)$.

\item[\rm(d)] $\D$ is one of two $2$-$(16,6,2)$ designs with $\Aut(\D)=2^4S_6$ and  $(\ZZ_2\times \ZZ_8)(S_4.2)$, respectively, and $G$ is either point-imprimitive or point-primitive of affine type.

\item[\rm(e)] $G\leq \AGammaL_1(p^d)$, for some odd prime $p$ and $|G|$ is odd.
\end{enumerate}
Further, if $G$ is point-imprimitive, then $(v,k,\lambda)=(16,6,2)$, as stated in  {\rm(d)}.
\end{lemma}

The following lemma is obtained directly using {\sc Magma} \cite{magma}, according to Lemma \ref{LemmaLambda=2}.

\begin{lemma}\label{LemmaRegueiroLambda=2}
Only one $2$-$(16,6,2)$ design $\D$ admits a flag-regular automorphism group.
This design has full automorphism group $\Aut(\D)=\ZZ_2^4{:}\SS_6$, and has exactly five non-isomorphic flag-regular, point-imprimitive automorphism groups. There are two groups $\ZZ_2^2{:}\SS_4$ and $(\ZZ_4^2{:}\ZZ_3){:}\ZZ_2$ admitting stabilizers $G_x\cong G_B\cong \DD_6$, and three groups $(\ZZ_2^4{:}\ZZ_2){:}\ZZ_3$, $\ZZ_2^3.\AA_4$ and $(\ZZ_4{:}\DD_8){:}\ZZ_3$ admitting stabilizers $G_x\cong G_B\cong \ZZ_6$.
\end{lemma}

For a given permutation group $G$ on a set $\P$ of degree $v$ and a parameter set $(v, b, r, k,\lambda)$, we may apply the following procedure to verify if there exists a 2-design with these parameters that admits $G$ as a flag-transitive (hence locally transitive) automorphism group.

Step 1. Find conjugacy classes of subgroups of index $b$ in $G$.

Step 2. Take a representative $H$ from such a conjugacy class as a candidate for a block stabilizer.

Step 3. Find a point-orbit $O$ of $H$ of length $k$, as a candidate for a block.

Step 4. Verify if the cardinality of $O^G$ equals $b$.

Step 5. Check whether the incidence structure $(\P,O^G)$ is a 2-design.

This procedure can be implemented using {\sc Magma}, and will be applied in the proof of the main theorems for some sporadic cases.




\medskip

\section{Local actions of automorphism groups for symmetric designs}\label{SectionLocalAction}

In this section we obtain some results concerning the local actions of automorphism groups on symmetric designs, which will be useful in the proof of the main theorems.

We first present the following Lemma \ref{LemmaAffineSoluble}, which treats solvable affine groups.
The idea of the proof is due to Dempwolff \cite{DempwolffSymaffine3} and Kantor \cite{Kantor2tr}. For the completeness, a proof is given here.

\begin{lemma}\label{LemmaAffineSoluble}
Let $\D$ be a symmetric design with $G \leqslant \Aut(\D)$.
If $G$ is point-primitive and solvable, then for any $B\in\B$ there exists $x\in\P\setminus B$ such that $G_B=G_x$.
\end{lemma}

\begin{proof}
According to the O'Nan-Scott theorem (Lemma \ref{LemmaONan}), only primitive groups of affine type are solvable.
In this case, $\P$ can be identified with $V=\ZZ_p^d$, $G$ is a subgroup of the affine group ${\rm AGL}_d(p)$
with $V$ the translation group and $G_x=G\cap \GL_d(p)$ irreducible on $V$.
As $G$ is solvable, $G_x$ is also solvable and has a non-trivial Fitting subgroup $F(G_x)$.
Since $G_x$ is irreducible on $V$, $O_p(G_x)=1$. This implies that $F(G_x)$ is a $p'$-group.
Note that $V=C_V(F(G_x))\times[F(G_x),V]$ is a $G_x$-decomposition.
From $[F(G_x),V]>1$ we have $[F(G_x),V]=V$ and thus $C_V(F(G_x))=1$.
Now $[N_V(F(G_x)), F(G_x)]\leq F(G_x)\cap V=1$ and we have $N_G(F(G_x))\cap V=N_V(F(G_x))=C_V(F(G_x))=1$.
Thus $N_G(F(G_x))=G_x$.

Let $L=VF(G_x)$. As $G_B$ is also a complement of $V$ in $G$, $G=VG_B$.
Let $L_0=G_B\cap L$ and then $L=L\cap G=VL_0$. We also have $L_0$ is a $p'$-group.
By the Schur-Zassenhauss theorem, replace $L_0$ by a suitable conjugate we have $L_0=F(G_x)$.
Therefore $G_B=N_G(L_0)=N_G(F(G_x))=G_x$ and we are done.
\end{proof}

\begin{lemma}\label{LemmaG^Bfaithful}
Let $G$ be a block-transitive automorphism group of a symmetric design $\D$ with $\lambda>1$.
If each non-identity element of $G_B^B$ fixes at most $\lambda$ points on $B$, then either $G_B$ acts faithfully on $B$, or ${\rm Fix}_{\P}(g)=B$ for each $g\in G_{(B)}$.
If each non-identity element of $G_B^B$ fixes at most $\lambda-1$ points on $B$, then $G_B$ acts faithfully on $B$.
\end{lemma}

\begin{proof}
Assume that $G_B$ acts unfaithfully on $B$. Equivalently, $G_{(B)}>1$. Let $1 \ne g\in G_{(B)}$.
Now $g$ is an axial automorphism with axis $B$.
If each non-identity element of $G_B^B$ fixes at most $\lambda$ points, we show that ${\rm Fix}_{\P}(g)=B$.
Assume that ${\rm Fix}_{\P}(g) \ne B$. Then $g$ fixes a point $x\in\P\setminus B$.
By Lemma \ref{LemmaAxisOutsideCentre}, $x$ is a centre of $g$.
Let $B_1$ be a block through $x$, which is fixed by $g$.
Then $|B\cap B_1|=\lambda$ and so each point of $B\cap B_1$ is fixed by $g$.
This follows that $g$ fixes at least $\lambda+1$ points on its fixed block $B_1$.
By the assumption and the block-transitivity, each non-identity element of $G_{B_1}^{B_1}$ fixes at most $\lambda$ points on $B_1$, so $g$ fixes every point of $B_1$.
Hence, $B_1$ is also an axis of $g$. This contradicts Lemma \ref{LemmaOneAxis}.

Now consider the case that each non-identity element of $G_B^B$ fixes at most $\lambda-1$ points.
If $G_B$ is not faithful on $B$, then let $g\in G_{(B)}>1$.
By Lemma \ref{LemmaSameFixed}, $g$ has at least $k$ fixed points and so has at least $k$ fixed blocks.
Let $B_2$ be a fixed block other than $B$. Now $B\cap B_2$ contains $\lambda$ fixed points of $g$. This implies that $B_2$ is also an axis of $g$ and we get the same contradiction.
\end{proof}

\begin{lemma}\label{LemmaG^Batmost2}
Let $G$ be a flag-transitive automorphism group of a symmetric design $\D$ with $\lambda>1$. If each non-identity element of $G_B^B$ fixes at most two points on $B$, then $G_B$ acts faithfully on $B$.
\end{lemma}

\begin{proof}
If $\lambda\geq3$, then the result follows from Lemma \ref{LemmaG^Bfaithful}.
If $\lambda=2$, then $\D$ is a biplane.
By the classification of flag-transitive biplanes (Lemma \ref{LemmaLambda=2}), either $G$ is a solvable affine group or $\D$ has parameters $(v,k,\lambda)=(4,3,2)$, $(7,4,2)$, $(11,5,2)$ or $(16,6,2)$.
For the first case, by Lemma \ref{LemmaAffineSoluble} we conclude that there exists $x\in\P$ such that $G_x=G_B$, where $B$ is a non-trivial orbit of a point stabilizer $G_x$.
Since $G$ has a normal regular subgroup, it follows from \cite[Exercise 4.4.8]{Dixon} that $G_x$ acts faithfully on each of its non-trivial orbits.
For the latter four cases, check their full automorphism groups listed in Lemma \ref{LemmaLambda=2} and we find that both local actions are faithful.
\end{proof}

\begin{lemma}\label{Lemma|H|lambda}
	Let $G$ be a block-transitive automorphism group of a symmetric design $\D$ with $\lambda>1$.
	If each non-identity element of $G_B^B$ fixes at most one point on $B$, and the subgroup $H$ of $G$ fixes at least two points on $\P$, then $|H|$ divides $\lambda$.
\end{lemma}

\begin{proof}
	
	Suppose $H$ is a subgroup of $G$ and fixes two points, say $x$ and $y$.
	Then $H$ fixes the set $S$ of $\lambda$ blocks containing $x$ and $y$.
	If there exists $1\neq h\in H$ such that $h$ fixes some block of these, say $B$, then $h\in H_B\leq G_B$ and $h$ fixes more than one point on $B$.
	By the assumption, the image of $h$ under the homomorphism from $G_B$ to $G_B^B$ is the identity.
	This means $h\in G_{(B)}$.
	Since each non-identity element of $G_B^B$ fixes at most one  point on $B$ and $1\leqslant \lambda -1$, by Lemma \ref{LemmaG^Bfaithful}, $G_B$ acts faithfully on $B$. Hence $h\in G_{(B)}=1$, a contradiction.
	It follows that each element of $H$ fixes no element of $S$, and so $H$ acts semi-regularly on $S$.
	Thus $|H|$ divides $|S|=\lambda$.
\end{proof}

We denote by $\lfloor a\rfloor$ the largest integer but not larger than $a$.
In the following we give a result concerning the number of fixed points of automorphism of order greater than or equal to $\lambda$.

\begin{lemma}\label{LemmaFixedPointsCore}
Let $G$ be a block-transitive automorphism group of a symmetric design $\D$ with $\lambda>1$. If $G_B^B$ is semi-regular on $B$, and $g\in G$ with $o(g)\geq\lambda$, then one of the following holds:
\begin{enumerate}
\item[\rm (a)] $|{\rm Fix}_{\P}(g)|\leq min\{\lfloor\frac{k}{\lambda}\rfloor \lfloor\frac{k-1}{\lambda}\rfloor+1,\,\lfloor\frac{k}{\lambda}\rfloor(\lambda-1)+1\}$, if $o(g)=\lambda$.

\item[\rm (b)] $|{\rm Fix}_{\P}(g)|\leq 1$, if $o(g)>\lambda$.
\end{enumerate}
\end{lemma}

\begin{proof}
{\rm (a)}
Let $g\in \Aut(\D)$ with $o(g)=\lambda$. If $|{\rm Fix}_{\P}(g)|\leq2$, then the lemma holds clearly. In the following we assume that $|{\rm Fix}_{\P}(g)|\geq3$.
Note that $g$ fixes an equal number of points and blocks (Lemma \ref{LemmaSameFixed}). We alternatively analyze the set $S$ of fixed blocks of $g$.
We claim that any three fixed blocks intersect either in a set of $\lambda$ points or in an empty set.
In fact, let two of them intersect in a set $\Delta$ of $\lambda$ points. If the third block $C$ satisfies $|C\cap\Delta|>1$,
then $g$ set-wisely fixes $C\cap\Delta$.
Since $G_C^C$ is semi-regular on $C$, each non-identity element of $G_C^C$ has no fixed points on $C$. By Lemma \ref{LemmaG^Bfaithful}, $G_C$ acts faithfully on $C$. Then $\langle g\rangle \leq G_C\cong G_C^C$ is semi-regular on $C$.
Clearly, each non-identity element of $\langle g\rangle$ has no fixed point on any block of $S$, and hence on the intersect $C\cap\Delta$ of three blocks.
So $\lambda=o(\langle g\rangle)$ divides $|C\cap\Delta|$.
Hence $C\cap\Delta=\Delta$ and the claim is proved.

Now let $B$ be a fixed block of $g$.
Then for blocks $C_i, C_j \in S \setminus \{B\}$, we have $(B \cap C_i) \cap (B \cap C_j)=B \cap C_i$ or $\emptyset$.
Define $$R=\{(C_i,C_j) \in (S \setminus \{B\}) \times (S \setminus \{B\}) |B \cap C_i=B \cap C_j\} $$
to be an equivalence relation on $ S \setminus \{B\}$. Let $h$ be the number of classes of $S \setminus \{B\}$ under the equivalence relation $R$. And without loss of generality, let $C_1, \dots, C_h \in S \setminus \{B\}$ such that $B \cap C_i \ne B \cap C_j$ for any $1 \leq i <j \leq h$.
Since $(B \cap C_1) \sqcup \dots \sqcup (B \cap C_h) \subseteq B$, we have $h \leq \lfloor\frac{k}{\lambda}\rfloor.$
Let $$S_i=\{C \in S \setminus \{B\}|B \cap C = B \cap C_i\}$$ and $m_i=|S_i|$.
Then we have $S= S_1 \sqcup \dots \sqcup S_h\sqcup \{B\}$.

Now consider all points incident with $S_i \sqcup \{B\}$. For any distinct blocks $B_1$ and $B_2$ of $S_i$, we have $B \cap B_1 = B \cap B_2=B\cap C_i$ and $(B_1 \setminus C_i) \cap (B_2 \setminus C_i)= \emptyset$. Hence all points incident with $S_i \sqcup \{B\}$ are $(B_1 \setminus C_i) \sqcup \dots \sqcup (B_{m_i} \setminus C_i) \sqcup ((B \setminus C_i) \sqcup C_i)$, where $B_j \in S_i$.
Then we have
$$(m_i+1)(k-\lambda)+\lambda\leq v=\frac{k(k-1)}{\lambda}+1.$$
So
$$m_i+1\leq\frac{k(k-1)-\lambda(\lambda-1)}{\lambda(k-\lambda)}=\frac{k+\lambda-1}{\lambda}=\frac{k-1}{\lambda}+1.$$
This gives
$$m_i\leq\lfloor\frac{k-1}{\lambda}\rfloor.$$

On the other hand, $|B \cap C_i|=\lambda \geq 2$.
By the definition of designs, the number of blocks incident with two points is $\lambda$, we have $m_i+1=|S_i \sqcup \{B\}| \leq\lambda$.
Hence
$$m_i\leq min\{\lfloor\frac{k-1}{\lambda}\rfloor,\lambda -1\}.$$

Since $S= S_1 \sqcup \dots \sqcup S_h\sqcup \{B\}$, we have
\begin{align*}
|{\rm Fix}_{\P}(g)|=|{\rm Fix}_{\B}(g)|&=m_1+m_2+\cdots+m_h +1\\
&\leq\lfloor\frac{k}{\lambda}\rfloor(min\{\lfloor\frac{k-1}{\lambda}\rfloor,\lambda-1\})+1\\
&=min\{\lfloor\frac{k}{\lambda}\rfloor\lfloor\frac{k-1}{\lambda}\rfloor+1,\,\lfloor\frac{k}{\lambda}\rfloor(\lambda-1)+1\}.
\end{align*}

{\rm (b)} Assume that $|{\rm Fix}_{\P}(g)|=|{\rm Fix}_{\B}(g)|> 1$. Clearly, by Lemma \ref{LemmaG^Bfaithful}, $\langle g\rangle$ acts faithfully and semi-regularly on the intersection $\Delta$ of any two fixed blocks. The contradiction arises from $o(\langle g\rangle)(>\lambda)$ dividing $|\Delta|=\lambda$ by Lemma \ref{Lemma|H|lambda}.
\end{proof}

\begin{remark}\rm
The dual statements of Lemmas \ref{LemmaAffineSoluble}--\ref{LemmaFixedPointsCore} also hold (in terms of point stabilizers), since the proofs can be applied to the dual structures of symmetric designs.
\end{remark}

\medskip

\section{Locally abelian symmetric designs and Frobenius groups}\label{SectionAbelianFrobenius}

In order to study locally transitive block designs with particular local actions, one might start from the cyclic or abelian local action.
Noticing that a primitive group with non-trivial abelian point stabilizers must be a Frobenius group,
we present results concerning locally abelian symmetric designs and Frobenius groups in the following Propositions \ref{LemmaAbelian} and \ref{LemmaFrobeniusFlagregular}.

\begin{proposition}\label{LemmaAbelian}
If $G$ is a locally transitive automorphism group of a symmetric design $\D$, where $G_B^B$ is abelian, then this local action is faithful, and there exists $x\in\P\setminus B$ such that $G_B=G_x$.
\end{proposition}

\begin{proof}
Since $G$ is locally transitive, $G_B^B$ as a homomorphic image of $G_B$ is abelian and transitive by Lemma \ref{LemmaFlagtrProperties} (b).
As a transitive abelian permutation group $H$ is regular, $G_B^B$ is regular.
If $\lambda=1$, then $\D$ is a projective plane, and by Lemma \ref{LemmaKantorProPlane} we know that $G$ is a solvable Frobenius group.
Thus the proposition holds by Lemma \ref{LemmaAffineSoluble}. Now we assume that $\lambda>1$.
It then follows from Lemma \ref{LemmaG^Batmost2} that $G_B$ is faithful on $B$.
Thus $|G_B|=|G_B^B|=k$.
It is clear that there exists a Sylow $p$-subgroup $L$ of $G_B$ such that $|L|\nmid\lambda$.
Otherwise, if for each Sylow $p$-subgroup $L$ of $G_B$ we have $|L|\mid\lambda$. Then it would deduce that $|G_B|=k\mid\lambda$, which contradicts $k>\lambda$.

Now, by Lemma \ref{Lemma|H|lambda} we get that $L$ fixes at most one point on $\P$.
If $L$ has no fixed points on $\P$, then all point-orbits of $L$ have lengths divisible by $p$, and it implies that
$$p\,\,\Big|\,\, v=\frac{k(k-1)}{\lambda}+1.$$
Since $|L|=p^i\mid k$ but $|L|\nmid\lambda$, we have $p\mid\frac{k(k-1)}{\lambda}$ and it implies that $p\mid1$, a contradiction.
Thus $L$ fixes exactly one point on $\P$, say $x$.
Note that $L\unlhd G_B$ as $G_B\cong G_B^B$ is abelian.
It follows that for any $g\in G_B$,
$$x^g=({\rm Fix}_{\P}(L))^g={\rm Fix}_{\P}(L^g)={\rm Fix}_{\P}(L)=x.$$
We get $G_B\leq G_x$ and thus $G_B=G_x$. The transitivity of $G_x^B=G_B^B$ implies that $x\in\P\setminus B$.
\end{proof}

A result of Kantor\cite{KantorProplane} (or see Lemma \ref{LemmaKantorProPlane}) showed that if a Frobenius group $G$ acts flag-transitively on a projective plane $\D$, then the number of points of $\D$ is a prime and $G$ has odd order.
The lemma below extends this result from projective plane to general symmetric designs.

\begin{proposition}\label{LemmaFrobeniusFlagregular}
If $G$ is a Frobenius group on $\P$, acting flag-transitively on a symmetric design $\D=(\P,\B)$, then $G\leq\AGammaL_1(v)$ and has odd order, and is point-primitive, block-primitive and flag-regular.
Moreover, for any $B\in\B$ there exists $x\in\P\setminus B$ such that $G_B=G_x$.
\end{proposition}

\begin{proof}
Since $G$ is a Frobenius permutation group, we have $G = N {:} H$, where the Frobenius kernel $N$ is a nilpotent regular normal subgroup, and the Frobenius complement $H$ satisfies $(|N|,|H|)=1$.
Note that the point stabilizers and the block stabilizers are complements of $N$ on $G$, and all complements of $N$ are conjugate.
Hence, the point stabilizers and block stabilizers lie in the conjugacy class of Frobenius complements.

Now the flag-transitivity of $G$ implies that any block of $\D$ is a suborbit of $G_x$ for some $x\in\P\setminus B$.
We get that $G_B=G_x$ acts regularly on $B$ and thus $G$ is flag-regular by Lemma \ref{LemmaDefFlagregular}.
We also deduce that $k=|B|$ divides $v-1$. From Lemma \ref{LemmaParameters} we have $\gcd(k,\lambda)=1$, and this follows that $G$ is point-primitive by \cite[2.3.7]{Dembowski}.
As $G$ has equivalent actions on $\B$ and $\P$, $G$ is also block-primitive.

Since $G$ admits a nilpotent regular normal subgroup $N$, by the O'Nan-Scott Theorem, $G$ can only be of affine type, whose socle is exactly the Frobenius kernel $N$ and is elementary abelian.
Now $\D$ induces an elementary abelian difference set.
Suppose that $|N|=v$ is even. Then $v$ is a power of 2.
By \cite[Theorem 1]{MannDiffenceSet}, $\mathcal{D}$ is a $2$-$(2^{2m},2^{2m-1}-2^{m-1},2^{2m-2}-2^{m-1})$ design or its complement design, which contradicts the fact that $k\mid v-1$ since $v$ and $k$ are even in either case.
So $N$ has odd order.
On the other hand, Lemma \ref{LemmaInvolutionFixSym} shows that if any involution of $G$ fixes some point, then it must fix more than one point. Hence the Frobenius complement of $G$ contains no involutions and so it has odd order. Thus $G$ has odd order.
Further, by \cite{Symrm=1Affine}, we know that $G\leq\AGammaL_1(v)$ for some prime power $v=p^d$ .
\end{proof}

Note that a transitive abelian permutation group is regular.
If a locally transitive automorphism group $G$ of a symmetric design $\D$ has an abelian point stabilizer $G_x$ , then $G_x^{\D(x)}$ is also abelian, and thus regular and faithful by the dual of Lemma \ref{LemmaG^Batmost2}.
Similarly, if $G$ has an abelian block stabilizer $G_B$, then $G_B^B$ is also abelian, and hence regular and faithful by the same lemma.
Apply the proof of Proposition \ref{LemmaAbelian} to the dual structure of $\D$, the dual statement of Proposition \ref{LemmaAbelian} holds for abelian $G_x^{\D(x)}$.
Hence we have the following.

\begin{proposition}\label{PropLocallyAbelian}
Let $\D=(\mathcal{P}, \mathcal{B})$ be a symmetric design and $G$ be a locally transitive automorphism group of $\D$.
Let $x\in\P$ and $B\in\B$.
If $K$ is an abelian group, then the following are equivalent:
\begin{enumerate}
\item[\rm(a)]  $G_x\cong K$.

\item[\rm(b)] $G_B\cong K$.

\item[\rm(c)] $G_x^{\D(x)}\cong K$.

\item[\rm(d)] $G_B^B\cong K$.

\item[\rm(e)] The block $B$ is a faithful orbit of some point stabilizer isomorphic to $K$.
\end{enumerate}
\end{proposition}

Notice that if $|G_x|=|G_B|$ is a prime $p$, then $k=p$.
It follows that $(k,\lambda)=1$ as $k>\lambda$. By a result of Dembowski \cite[2.3.8]{Dembowski}, $G$ must be a primitive Frobenius group.
By Proposition \ref{LemmaFrobeniusFlagregular} we immediately have the fact as follows.

\begin{corollary}\label{PropoGB=p}
If $G$ is a locally transitive automorphism group of a symmetric design $\D$ with point stabilizers of prime order $p$, then $G$ is a primitive Frobenius group of odd order.
\end{corollary}

\medskip

\section{Locally dihedral symmetric designs}\label{SecProof1}

Now we focus on the locally dihedral symmetric designs.
For the convenience, we make the following Hypothesis in this section.

\begin{hypothesis}
Suppose that $\D=(\mathcal{P},\mathcal{B})$ is a symmetric $2$-$(v,k,\lambda)$ design and $G$ is a locally transitive
(equivalently, flag-transitive) automorphism group of $\D$.
\end{hypothesis}


The following Lemmas \ref{Lemma1orF} and \ref{LemmaDihedralAction} are folklore and elementary.

  \begin{lemma}\label{Lemma1orF}
If $G_B^B=N{:}F\cong \DD_{2n}$ where $N\cong \ZZ_n$ and $F\cong \ZZ_2$, then for any $x\in B$, $(G_B^B)_x$ is either trivial or conjugate to $F$, namely, $k=2n$ or $n$.
\end{lemma}

\begin{proof}
The flag-transitivity of $G$ and Lemma \ref{LemmaFlagtrProperties} (b) implies that $G_B^B$ is transitive.
The conclusion immediately follows from the fact that the only core-free subgroups of $G_B^B$ are the trivial subgroup and the conjugates of $F$.
\end{proof}

\begin{lemma}\label{LemmaDihedralAction}
Suppose that $G_B^B=N{:}F\cong \DD_{2n}$ and $(G_B^B)_x=F$ for any $x\in B$. If $|N|$ is odd, then each non-identity element of $G_B^B$ fixes at most one point on $B$. If $|N|$ is even, then each non-identity element of $G_B^B$ either fixes no points or fixes exactly two points on $B$.
\end{lemma}

\begin{proof}
If $|N|$ is odd, then $G_B^B$ is a Frobenius group with the Frobenius kernel $N$. Moreover, $G_B^B$ has exactly one conjugacy class of subgroups of order 2.

If $|N|$ is even, then $G_B^B$ has exactly 3 conjugacy classes of subgroups of order 2. One such class is the center ${\rm Z}(G_B^B)$, containing the unique subgroup of order 2 in $N$.
The non-identity elements in the other two classes form a partition of $G_B^B\setminus N$.
Let $F_1$ and $F_2$ be the representatives of these two classes respectively.
Since $(G_B^B)_x$ is core-free, we have $(G_B^B)_x$ is conjugate to $F_1$ or $F_2$.
The actions of $G_B^B$ on the coset space $[G_B^B{:}F_1]$ and on the coset space $[G_B^B{:}F_2]$ are permutation isomorphic.
In these two actions, each non-identity element of $G_B^B$ either fixes no points or fixes exactly two points.
\end{proof}

\begin{lemma}\label{LemmaLambda=1}
 If $G_B^B$ is dihedral and $\lambda=1$, then $\D$ is the Fano plane and $G=\PSL_3(2)$ with $G_B^B\cong \DD_6$, in which case $G_B\cong \SS_4$ acts unfaithfully on $B$.
 Moreover, if $G_B$ is dihedral, then $\lambda>1$.
\end{lemma}

\begin{proof}
Assume that $G_B^B\cong \DD_{2n}$ and $\lambda=1$.
By Lemma \ref{LemmaKantorProPlane}, if a projective plane admits a flag-transitive automorphism group $G$, then either $\mathcal{D}$ is Desarguesian or $G$ is a Frobenius group of odd order.
The latter case is impossible since $G$ has even order.
So $\mathcal{D}$ is Desarguesian with $\PSL_3(q)\unlhd G\leq \PGammaL_3(q)$, where $q$ is a prime power $p^d$.
Note that the finite Desaguesian projective plane is isomorphic to the point-hyperplane design of projective space $\PG(2,q)$.
The group $\PSL_3(q)$ is 2-transitive on the points of $\PG(2,q)$, which implies that $G_B^B\cong \DD_{2n}
\geq\PSL_3(q)_B^B$ is also 2-transitive.
The only dihedral group admitting a faithful 2-transitive representation is $\DD_6$.
In this case $q=2$ and $\D$ is $\PG(2,2)$ (the unique $2$-$(7,3,1)$ design), $G=\PSL_3(2)$ with $G_B^B\cong \DD_6\cong S_3$ and $G_B\cong \SS_4$ acting unfaithfully on $B$.

If $G_B$ is dihedral,
then $G_B^B$ is either dihedral or has order 2.
It follows from the argument above and the transitivity of $G_B^B$ that $\lambda>1$.
\end{proof}

  \begin{hypothesis}
In the remaining part of this section we always assume that $\lambda>1$.
\end{hypothesis}

\begin{lemma}\label{LemmaFaithfulG^B}
If $G_B^B$ or $G_B$ is a dihedral group, then $G_B$ acts faithfully on $B$.
\end{lemma}

\begin{proof}
For the case $G_B^B$ is a dihedral group, the lemma immediately follows from Lemma \ref{LemmaDihedralAction} and \ref{LemmaG^Batmost2}.
If $G_B\cong \DD_{2n}$ and  is not faithful on $B$, then $G_B^B$ is isomorphic to a quotient group of $G_B\cong \DD_{2n}$. Thus $G_B^B$ is either dihedral or of order 2.
The transitivity of $G_B^B$ implies that $|G_B^B|\neq2$. Thus the lemma holds.
\end{proof}

\begin{lemma}\label{LemmaG_B=Gx}
If $G_B=N{:} F\cong \DD_{2n}$, then there exists $x\in\P\setminus B$ such that $G_B=G_x$.
Further, $|{\rm Fix}_{\P}(N)|\in\{1,3\}$, and if $|{\rm Fix}_{\P}(N)|=3$, then $(v,k,\lambda)=(4n-1,2n,n)$.
\end{lemma}

\begin{proof}
Consider the fixed points of $N\unlhd G_B=N{:} F$.
Since $N$ is cyclic, $N$ has the same fixed points and fixed blocks as its generator.
By Lemma \ref{LemmaSameFixed}, $N$ fixes an equal number of points and blocks.
As $N$ fixes block $B$, $N$ fixes at least one point, say $x$.
Since $|G_x{:}N|=|G_B{:}N|=2$, we have
$N\unlhd G_x$ and $N\unlhd G_B$. For any $g\in G_x$ or $g\in G_B$, we have
$$({\rm Fix}_{\P}(N))^g={\rm Fix}_{\P}(N^g)={\rm Fix}_{\P}(N),$$
Hence $G_x$ and $G_B$ set-wisely fixes ${\rm Fix}_{\P}(N)$.

If $N$ fixes exactly one point $x$, then $G_B$ fixes $x={\rm Fix}_{\P}(N)$.
This follows that $G_B\leq G_x$. As $|G_B|=|G_x|$, we have $G_B=G_x$ and the lemma follows.

Suppose $N$ fixes more than one point. Then $N$ fixes at least two blocks.
Let $B'$ be a fixed block of $N$ other than $B$.
Now $N$ set-wisely fixes $B\cap B'$ and $N\leq G_{B'}$.
Note that $G_{B'}$ is conjugate to $G_B\cong \DD_{2n}$ in $G$.
So $N$ is also a normal, semi-regular subgroup of $G_{B'}$ on $B'$.
This implies that $|N|$ divides $|B\cap B'|=\lambda$.
From Lemma \ref{Lemma1orF} we get that $n=|N|=k$ or $\frac{k}{2}$.
Combining this with $k>\lambda$, we have $n=|N|=\frac{k}{2}$ and so $\lambda=\frac{k}{2}$.
Then $v=\frac{k(k-1)}{\lambda}+1=4n-1$.
In this case, $|G_B|=2|N|=k$ and Lemma \ref{LemmaDefFlagregular} implies that $G$ is flag-regular.
Now apply Lemma \ref{LemmaFixedPointsCore} and we conclude that $N$ fixes exactly two or three points.

If $N$ fixes exactly two points, say $x$ and $y$, then $G_x$ set-wisely fixes ${\rm Fix}_{\P}(N)=\{x,y\}$ and thus $G_x$ also fixes $y$.
We have $G_x=G_y=G_x\cap G_y$.
By Lemma \ref{LemmaFlagtrProperties} (c), the flag-transitivity implies that $G_x$ acts transitively on the blocks through $x$.
But $G_x$ set-wisely fixes the $\lambda$ blocks through $x$ and $y$, which forces $k=\lambda$, a contradiction.

Hence, $N$ fixes exactly three points, say $x$, $y$ and $z$, then similar to the argument above, $G_B$ set-wisely fixes $\{x,y,z\}$.
As the complement $F$ of $N$ in $G_B$ is an involution, $F$ fixes either one point or all three points in $\{x,y,z\}$.
If $F$ fixes each one of $\{x,y,z\}$, then $G_B=N{:} F$ fixes each one of $\{x,y,z\}$.
We deduce that $|G_B|=k$ divides $\lambda$ by Lemma \ref{Lemma|H|lambda}, a contradiction.
If $F$ fixes exactly one point in this set, say $x$, then $G_B=N{:} F$ fixes $x$.
This forces $G_B\leq G_x$ and thus $G_B=G_x$, so the proof is completed.
\end{proof}

\begin{lemma}\label{LemmaLambdaEven}
If $G_B=N{:} F\cong \DD_{2n}$, then $\lambda$ is even.
\end{lemma}

\begin{proof}
Suppose that $|N|=n$ is odd. According to Lemma \ref{LemmaDihedralAction}, each  non-identity element of $G_B^B$ fixes at most one point on $B$.
The involution $g$ in $F$ fixes $B$, by Lemmas \ref{LemmaSameFixed} and \ref{LemmaInvolutionFixSym}, $g$ fixes at least two blocks (and two points).
From Lemma \ref{Lemma|H|lambda} we have $|F|=2$ divides $\lambda$.

Suppose that $|N|=n$ is even.
There is an involution $g$ in $N$, which acts semi-regularly on $B$.
Similarly, $g$ fixes at least two blocks. Let $B'$ is a fixed block of $g$ other than $B$.
Then $\langle g\rangle$ acts semi-regularly on $B\cap B'$, which follows that $|\langle g\rangle|=2$ divides $|B\cap B'|=\lambda$.
\end{proof}

In the following we deal with the case that $G$ is point-primitive.

\begin{lemma}\label{LemmaGaffine}
If $G_B=N{:} F\cong \DD_{2n}$ and $G$ is point-primitive, then $G$ is of affine type.
\end{lemma}

\begin{proof}
By Lemma \ref{LemmaG_B=Gx}, we know that $G_x=G_B\cong \DD_{2n}$ for some $x\in\P\setminus B$.
Then we only need to show that the socle $T$ of $G$ is not non-abelian simple by  Lemma \ref{LemmaDihedralAffAlmost}.
If $T$ is non-abelian simple, then according to Theorem \ref{ProLCH} and Remark \ref{RemarkOfTable1}, $T$ is isomorphic to $\PSL_2(q)$ or $\Sz(q)$.
Flag-transitive, point-primitive symmetric designs with these two types of socles are classified recently.
If $T=\PSL_2(q)$, then the groups with such socle and the corresponding point stabilizers are listed in \cite[Table 1]{AlaviPSL4}. None of these has a dihedral point stabilizer.
For $T=\Sz(q)$,
it is also shown in \cite{AlaviExceptional} that no symmetric designs admit a flag-transitive, point-primitive automorphism group with socle of Suzuki type.
\end{proof}

%

\begin{lemma}\label{LemmaGaffineOrbitsNum}
Let $\D$ and $G$ be as in Lemma {\rm\ref{LemmaGaffine}}.
Then $N$ acts semi-regularly on $\P\setminus\{x\}$ and has odd order.
Moreover, if $m$ is the number of fixed points of $F$, and $a$, $c$ are the number of suborbits of length $n$ and $2n$, respectively,
then

$$
a = m-1,\qquad
c = \begin{cases}
\displaystyle\frac{k-1}{2\lambda} - \frac{m-1}{2}, & \text{if } k = n,\\[10pt]
\displaystyle\frac{k-1}{\lambda} - \frac{m-1}{2}, & \text{if } k = 2n.
\end{cases}
$$
\end{lemma}

\begin{proof}
As $G$ has a normal regular subgroup, from \cite[Exercise 4.4.8]{Dixon} we know that $G_x$ acts faithfully on each of its non-trivial orbits.
For any $y\in\P\setminus\{x\}$, $G_x\cap G_y$ is a core-free subgroup of $G_x$, so $G_x\cap G_y$ is either $\{1\}$ or conjugate to $F$.
This means each non-trivial orbit of $G_x$ has length $n$ or $2n$.
By Lemma \ref{LemmaInvolutionFixSym}, any involution in $N$ fixes more than one point, a contradiction.
So $N$ cannot contain any involution and thus $N$ has odd order.
Now $G_x=G_B\cong \DD_{2n}$ is a Frobenius group.
Note that each involution of $G_x$ fixes exactly one point on the orbits of length $n$.
This yields $a=m-1$ and
$$c=\frac{v-1-(m-1)n}{2n}=\frac{k(k-1)-\lambda(m-1)n}{2n\lambda}.$$
Combining this with $k=n$ or $2n$ (see Lemma \ref{Lemma1orF}) and we have the result.
\end{proof}

\begin{lemma}\label{LemmaGknotn}
If $G_B=N{:} F\cong \DD_{2n}$, then $G$ is not point-primitive.
\end{lemma}

\begin{proof}
Assume that $G$ is point-primitive. By Lemma~\ref{LemmaGaffine}, $G$ is of affine type. Then $G=\ZZ_p^d{:}\DD_{2n}$, where $p$ is prime.
By Lemmas \ref{Lemma1orF} and \ref{LemmaGaffineOrbitsNum}, we have $k\in\{n, 2n\}$ and $n$ is odd.
Hence $4 \nmid k$.

We first treat the case $p=2$. As $G$ possesses an elementary abelian, point-regular normal subgroup, $\D$ naturally induces an elementary abelian difference set.
By a result \cite[Theorem 1]{MannDiffenceSet} on elementary abelian difference sets, we have that $\mathcal{D}$ is a $2$-$(2^{2m},2^{2m-1}-2^{m-1},2^{2m-2}-2^{m-1})$ design or its complement design with parameters $(2^{2m},2^{2m-1}+2^{m-1},2^{m-1}+2^{2m-2})$, where $m>1$. Hence $2^{m-1} \mid k= 2^{2m-1} \pm 2^{m-1}$. Since $4 \nmid k$, we have $m=2$, and $(v,k,\lambda)=(16,6,2)$ or $(16,10,6)$. Then $k=2n$ and $\D$ is $G$-flag-regular.
Check \cite[Table 9.62]{Handbook} for primitive groups of degree 16 and we see that $G=\ZZ_2^4{:}\DD_{10}$, and so $(v,k,\lambda)=(16,10,6)$.
This case is ruled out by using {\sc Magma}.

Now consider the case $p$ is odd.
If $k=n$, then
by Lemmas \ref{LemmaG_B=Gx} and \ref{LemmaGaffineOrbitsNum}, the block $B$ is an orbit of some point stabilizer $G_x$ of length $n$ with $k=n$ odd.
It then follows from Lemma \ref{LemmaInvolutionFixSym}~(b) that $|{\rm Fix}_{\P}(F)|\geq\frac{k-1}{\lambda}+1$, so
$$0\leq c=\frac{k-1}{2\lambda}-\frac{|{\rm Fix}_{\P}(F)|-1}{2}\leq \frac{k-1}{2\lambda}-\frac{k-1}{2\lambda}=0.$$
This means that $a=\frac{k-1}{\lambda}$ and $c=0$.
So all non-trivial orbits of $G_x$ have length $n$, which is odd.
Now $G$ is a $\frac{3}{2}$-transitive Frobenius regular group, where $G_x$ acts on each of its non-trivial orbits in its unique Frobenius representation.
Since $G$ is point-transitive and $|G|$ is even, by \cite[Exercise 3.13]{Wielandt} we get that $v=p^d$ is even, a contradiction.
If $k=2n$, then by Lemma~\ref{LemmaGaffineOrbitsNum}, we have
$$
v-1=p^d-1=a \cdot n+ c \cdot 2n=(a+2c)\cdot n.
$$
Since $p^d-1$ is even and $n$ is odd, we have $a+2c$ is even, and hence $a$ is even. Then
$$v-1=(\frac{a}{2}+c)\cdot 2n=(\frac{a}{2}+c)\cdot k$$ and $k \mid v-1$. By the equation $\lambda(v-1)=k(k-1)$, we have $\lambda \mid k-1=2n-1$. Hence $\lambda$ is odd, a contradiction to Lemma~\ref{LemmaLambdaEven}.
\end{proof}

\noindent{\bf Proof of Theorem \ref{ThmSym}}\,\,
First consider that $K$ is dihedral.
The equivalence of parts (b) and (d) follows from Lemma \ref{LemmaFaithfulG^B}.
Part (b) implying part (e) is from Lemma \ref{LemmaG_B=Gx} and the flag-transitivity of $G$.
Part (e) implying part (b) follows from $G_x\leq G_{\Delta}$ and $|G_x|=|G_{\Delta}|$, as $\Delta$ is a block.
Thus, (b), (d) and (e) are equivalent.
The equivalence of parts (a) and (e), also, (a) and (c) are directly obtained by applying the proof to the dual structures of
symmetric designs.
Hence, (a)--(e) are equivalent.
Further, if one of {\rm (a)--(e)} holds, then $G$ is imprimitive on $\P$ by Lemma \ref{LemmaGknotn}.
Notice that $G_x$ and $G_B$ are conjugate in $G$, so the action of $G$ on $\B$ is equivalent to that on $\P$, and hence $G$ is also imprimitive on $\B$.

If $K$ is abelian, then the equivalence of parts {\rm (a)--(e)} follows from Proposition \ref{PropLocallyAbelian}.
Assume that one of {\rm (a)--(e)} holds. Then $G_x$ and $G_B$ are conjugate in $G$, and $G_x\cong G_B\cong G_B^B\cong G_x^{\D(x)}\cong K$.
If $G$ is primitive on $\P$, then $G$ is also primitive on $\B$.
By the fact that a primitive group with non-trivial abelian point stabilizers is a Frobenius group,
$G$ is clearly a Frobenius group.
Further, the abelian stabilizer of a primitive group is cyclic (see \cite[Theorem 1]{LiebeckCyclicStabilizer}). The conclusion holds by Proposition \ref{LemmaFrobeniusFlagregular}.
\hfill$\square$

\medskip

\section{Point-locally dihedral, point-primitive non-symmetric 2-designs}\label{pri2designs}

In this section we tackle the non-symmetric 2-designs, on the point-primitive case.
We make the following hypothesis.

\begin{hypothesis}
Suppose that $\D=(\mathcal{P},\mathcal{B})$ is a non-symmetric design and $G$ is a locally-transitive
(equivalently, flag-transitive) and point-primitive automorphism group of $\D$, with point stabilizers a dihedral group.
\end{hypothesis}

Note that under this hypothesis, the local action $G_x^{\D(x)}$ must be a dihedral group, whether this action is faithful or not.
We shall tackle the classification for almost simple groups $G$.
The proof is completed by analyzing all the cases in Table \ref{TablePriGroupDihedralSta}.

\begin{lemma}\label{LemmaDesignNotSuzuki}
If $G$ is almost simple, then $G$ has socle $\PSL_2(q)$.
\end{lemma}

\begin{proof}
According to Theorem \ref{ProLCH} and Remark \ref{RemarkOfTable1}, if the socle of $G$ is not $\PSL_2(q)$, then $G=\Sz(q)$ with $G_x=\DD_{2(q-1)}$.
But $|\Sz(q)|=q^2(q^2+1)(q-1)$, which contradicts Lemma \ref{LemmaCubic}.
\end{proof}

\begin{lemma}\label{Lemmaqeven}
Assume that $G$ is almost simple with socle $\PSL_2(q)$.
If $q$ is even, then one of the following holds:
\begin{enumerate}
\item[\rm(a)] $\D$ is in the family of the Witt-Bose-Shrikhande space with $G=\PSL_2(q)\,(q\geq8)$ and $G_x=\DD_{2(q+1)}$.

\item[\rm(b)] $\D$ is a $2$-$(6,3,2)$ design with $G=\PSL_2(4)\cong\PSL_2(5)\cong \AA_5$ and $G_x=\DD_{10}$.

\item[\rm(c)] $\D$ is a $2$-$(28,7,2)$ design with $G=\PSL_2(8)$ and $G_x=\DD_{18}$.

\item[\rm(d)] $\D$ is a $2$-$(36,6,2)$ design with $G=\PSL_2(8)$ and $G_x=\DD_{14}$.
\end{enumerate}
\end{lemma}

\begin{proof}
Since $q$ is even, we have $G=\PSL_2(q)$ with $q=2^f>2$ and the dihedral stabilizer $G_x$ is a $\mathcal{C}_2$ or $\mathcal{C}_3$-subgroup in the Aschbacher-classes, as stated in Theorem \ref{ProLCH}.

By the classification of flag-transitive $2$-$(v,k,1)$ designs \cite{Buekenhout2vk1}, we observe that if  $G$ is almost simple with socle $\PSL_2(q)$ and with dihedral point stabilizer, then $G=\PSL_2(q)$ with $q\geq8$ even, and $\D$ is the Witt-Bose-Shrikhande space.
This corresponds to part (a). Hence we assume that $\lambda>1$ in the following.

Case 1. First suppose that $G_x=\DD_{2(q-1)}$ is a $\mathcal{C}_2$-subgroup.
Then $$v=\frac{|\PSL_2(q)|}{2(q-1)}=\frac{q(q+1)}{2}.$$
Consider the action of $G_x$ on $\mathcal{D}(x)$.
By the flag-transitivity of $G$ and Lemma \ref{LemmaFlagtrProperties} (c), $r=\frac{|G_x|}{|G_{xB}|}=\frac{2(q-1)}{|G_{xB}|}$ for some block $B$ through $x$.
Since $\lambda \geq 2$, by Lemma \ref{LemmaParameters} (d) we have
$$q(q+1)\leq\frac{1}{2}\lambda q(q+1)=\lambda v<r^2=\frac{4(q-1)^2}{|G_{xB}|^2}.$$
It is easy to see that $|G_{xB}|=1$ and $r=2(q-1)$.
Now we have $\lambda q(q+1)<8(q-1)^2$ and so $\lambda<8$.
From \cite[Table 2]{PGLSuborbits} we know that $\PSL_2(q)$ acts on the coset space of its maximal subgroup $\DD_{2(q-1)}$ with a subdegree $q-1$.
By Lemma \ref{LemmaDavies}, $r\mid\lambda(q-1)$.
So $2\mid\lambda$ and thus $\lambda\in\{2,4,6\}$.
If $\lambda=2$, then by the basic arithmetical properties in Lemma \ref{LemmaParameters}, we have $k=\frac{\lambda(v-1)}{r}+1$.
So
$(v,k,\lambda)=(\frac{q(q+1)}{2},\frac{q+4}{2},2)$.
By \cite[Table 2]{AlaviLambda2} and \cite[Theorem 1.1]{MontinaroLambda2}, we see that $\D$ is a $2$-$(36,6,2)$ design and $G=\PSL_2(8)$ with $G_x=\DD_{14}$, as in part (d).
If $\lambda=4$, then
$(v,k,\lambda)=(\frac{q(q+1)}{2},q+3,4)$ with
$$b=\frac{q(q+1)(q-1)}{q+3}=\frac{(q+3)(q^2-3q+8)-24}{q+3}=q^2-3q+8+\frac{-24}{q+3}.$$
We have a contradiction as $q+3\nmid 24$ as $q$ is even.
If $\lambda=6$, then similarly we get
$$b=\frac{2(q^3-q)}{3q+8}=\frac{2}{3}q^2-\frac{16}{9}q+\frac{110}{27}-\frac{880}{27(3q+8)}.$$
Multiplying both sides by $27$, we get that $3q+8\mid 880$.
We then obtain $q=4$ and $b=6$, but it follows that $b<v=10$, a contradiction.

Case 2. Now suppose that $G_x=\DD_{2(q+1)}$ is a $\mathcal{C}_3$-subgroup.
Then $$v=\frac{|\PSL_2(q)|}{2(q+1)}=\frac{q(q-1)}{2}.$$
If $q<8$, then $q=4$.
Then, $G=\PSL_2(4)\cong\PSL_2(5)\cong A_5$ and $G_x=\DD_{10}$. We have $(v,k,\lambda)=(6,3,2)$, and such a design exists by \cite[Table 1]{AlaviLambda2}, which corresponds to part (b).
Hence, $q\geq8$, and then $G$ acts on $\P$ as a $\frac{3}{2}$-transitive Frobenius regular group, where $G_x$ has $\frac{v-1}{q+1}+1=\frac{q}{2}$ point-orbits, and each non-trivial orbit has length $q+1$ (see \cite{CaminaMcDermott3/2}).
Hence each 2-element of $G_x$ has exactly $\frac{q}{2}$ fixed points on $\P$.
Before we proceed with the proof, we mention that all maximal subgroups of $G$ have types $\ZZ_2^f{:}\ZZ_{q-1}$, $\DD_{2(q-1)}$, $\DD_{2(q+1)}$, and $\PGL_2(q_0)$, where $q=q_0^r$, $q_0\ne 2$ and $r$ is prime (see \cite[Theorem 2.1]{MGiudici}).
It is known that $G$ acting on the right coset space $[G:(\ZZ_2^f{:}\ZZ_{q-1})]$ is a Zassenhaus group, namely, a 2-transitive group where each non-identity element has at most two fixed points.
Further, $\ZZ_2^f{:}\ZZ_{q-1}$ is a Frobenius group with $\ZZ_{q-1}$ acting transitively on the non-identity elements of $\ZZ_2^f$.
Since $\ZZ_2^f$ is a Sylow $2$-subgroup of $G$, all the 2-elements of $G$ are conjugate and of order 2, and thus are permutation isomorphic.
Therefore, each 2-element has exactly $\frac{q}{2}$ fixed points on $\P$.
Moreover, we notice that any two 2-elements from maximal subgroups of type $\ZZ_2^f{:}\ZZ_{q-1}$ (respectively, $\DD_{2(q-1)}$) cannot have a common fixed point on $\P$, otherwise these two elements generalize a subgroup $X$ fixing a point, and thus $\ZZ_2<X\leq(\ZZ_2^f{:}\ZZ_{q-1})\cap \DD_{2(q+1)}$ (respectively, $\ZZ_2<X\leq \DD_{2(q-1)}\cap \DD_{2(q+1)}$), a contradiction.
It follows that the fixed points of the 2-elements of maximal subgroups of type $\ZZ_2^f{:}\ZZ_{q-1}$ (respectively, $\DD_{2(q-1)}$) form a partition of $\P$.

Consider the action of $G_x=N{:}F$ on $\D(x)$.
We have
$$q(q-1) \leq\frac{1}{2}\lambda q(q-1)=\lambda v<r^2=\frac{4(q+1)^2}{|G_{xB}|^2}.$$
Since $q\geq 8$, we have $|G_{xB}|\in\{1,2\}$.

Subcase 1. Assume that $|G_{xB}|=2$.
In this case $r=q+1$.
If $k\leq q-1$, then
$$(q+1)(q-2)\leq2(\frac{1}{2}q(q-1)-1)=2(v-1)\leq\lambda(v-1)=r(k-1)\leq(q+1)(q-2).$$
It follows that $\lambda=2$ and $k=q-1$.
So $\D$ satisfies $(r,\lambda)=1$. This case is tackled in \cite{ZhanXQPSL}, and
only a $2$-$(28,7,2)$ design is obtained, where $G=\PSL_2(8)$ with $G_x=\DD_{18}$, as in part (c).
Now we consider the cases $k>q-1$.
Since $k\leq r=q+1$, we have $k\in\{q,q+1\}$.
If $k=q+1=r$, then $\D$ is symmetric, not satisfy the assumption.
Hence, $k=q$.
Then $$(q+1)(q-1)=r(k-1)=\lambda(v-1)=\lambda(\frac{1}{2}q(q-1)-1).$$
Clearly, $(\frac{1}{2}q(q-1)-1,q-1)=1$, so $\frac{1}{2}q(q-1)-1$ divides $q+1$, which contradicts $q\geq8$.

Subcase 2. Assume that $|G_{xB}|=1$.
In this case $r=2(q+1)$.
By $\lambda v<r^2$, we get
\begin{align}\label{inequality1}
\lambda<\frac{r^2}{v}=8\cdot\frac{q^2+2q+1}{q^2-q}=8\cdot(1+\frac{3q+1}{q^2-q}).
\end{align}
If $q=8$, then $G=\PSL_2(8)$ with $v=28$, $r=18$.
By Inequality (\ref{inequality1}), we get $\lambda\leq11$.
Combining this with Lemma \ref{LemmaParameters}, we find that $(v,k,\lambda)=(28,4,2)$ or $(28,7,4)$.
Using {\sc Magma} we see that $\PSL_2(8)$ cannot act as a flag-transitive automorphism group on a design with these parameters.
Therefore, $q>8$, so $q\geq16$.
By Inequality (\ref{inequality1}), we have $\lambda\leq9$.
Since $\lambda(v-1)=r(k-1)=2(q+1)(k-1)$ is even, $\lambda$ is also even.
So $\lambda\in\{2,4,6,8\}$.

If $\lambda=2$, then $k=\frac{q}{2}$, $b=2(q-1)(q+1)$, and $|G_B|=\frac{q}{2}$.
So $G_B$ lies in a Sylow 2-subgroup $Q\cong \ZZ_2^f$ of $G$.
Since the fixed points of the 2-elements of $Q$ form a partition of $\P$. For any point $x \in \P$, we have $Q_x\cong \ZZ_2$, and the orbit of $Q$ containing $x$ has length $|Q|/|Q_x|=\frac{q}{2}$.
Since each orbit of $Q$ on $\P$ has length $\frac{q}{2}$, it follows that $G_B=Q$, a contradiction.

If $\lambda=4$, then $k=q-1$, $b=q(q+1)$, and $|G_B|=q-1$.
Since all subgroups of order $q-1$ are conjugate to $\ZZ_{q-1}$ in $G$, it follows that $G_B$ lies in some maximal subgroup of type $\DD_{2(q-1)}$.
However, the fixed points of 2-elements of a subgroup of type $\DD_{2(q-1)}$ form a partition of $\P$, which yields that all point-orbits of subgroups of type $\DD_{2(q-1)}$ have length $q-1=k$. So $G_B=\DD_{2(q-1)}$, a contradiction.

If $\lambda=6$, then $k=|G_B|=\frac{3}{2}q-2$.
 Obviously, $G_B$ cannot be contained in a maximal subgroup of $G$ with structure $\DD_{2(q-1)}$ or $\DD_{2(q+1)}$ as any subgroups of these maximal subgroups cannot have an orbit of length $\frac{3}{2}q-2$.
 If $G_B\leq \ZZ_2^f{:} \ZZ_{q-1}$, then $\frac{3}{4}q-1$ divides $q-1$, which is impossible.
 Besides, $G_B$ is clearly not isomorphic to $\PGL_2(q_1)\cong \PSL_2(q_1)$ with $q=q_1^s$ for some $s$ and $q_1\ne2$.
Thus this case is ruled out.

Lastly, if $\lambda=8$, then $k=|G_B|=2q-3$.
It is also obvious that $G_B$ cannot be contained in any maximal subgroup of $G$.
\end{proof}

\begin{lemma}\label{Lemmaqodd}\label{LemmaDesignqOdd}
Assume that $G$ is almost simple with socle $\PSL_2(q)$ with $q$ odd.
Then $\D$ is a $2$-$(10,4,2)$ design with $G=\PGL_2(5)$ and $G_x=\DD_{12}$.
\end{lemma}

\begin{proof}
By Theorem \ref{ProLCH}, we know that $G$ is either $\PSL_2(q)$ or $\PGL_2(q)$.
If $q\neq5$, then for any $T=\PSL_2(q)$ with dihedral maximal subgroup $M_0$, there is $L=\PGL_2(q)=\PSL_2(q).2$ with dihedral maximal subgroup $M$, such that $M_0=M\cap T$.
Hence the action of $T$ on the coset space $[L:M]$ is equivalent to the action of $T$ on $[T:M_0]$ (the action is faithful).
This naturally induces
$$1<T<L\leq \Sym([L:M]).$$
Now, by Lemma \ref{LemmaLargegroup}, the non-existence of designs with point set $\P=[L:M]$ for $L$ will imply the non-existence for $T$.
For $q=5$, $T=\PSL_2(5)\cong\PSL_2(4)$, the case for $G=\PSL_2(4)$ is dealt with in Lemma \ref{Lemmaqeven}.
Further, we need not consider the case $\lambda=1$ as stated in the proof of Lemma \ref{Lemmaqeven}.
Therefore, we shall only assume that $G=\PGL_2(q)$ with dihedral maximal subgroups and $\lambda\geq2$ in the following.

First we consider the $\mathcal{C}_2$-subgroup $G_x=\DD_{2(q-1)}$.
Then $$v=\frac{|\PGL_2(q)|}{2(q-1)}=\frac{q(q+1)}{2}.$$
Consider the action of $G_x$ on $\mathcal{D}(x)$.
The flag-transitivity yields that $r=\frac{|G_x|}{|G_{xB}|}=\frac{2(q-1)}{|G_{xB}|}$ with some $B$ through $x$.
By Lemma \ref{LemmaParameters} (d) we have
$$q(q+1)\leq \frac{1}{2}\lambda q(q+1)=\lambda v<r^2\leq\frac{4(q-1)^2}{|G_{xB}|^2}.$$
It is easy to see that $|G_{xB}|=1$ and $r=2(q-1)$.
Now we have $\lambda q(q+1)<8(q-1)^2$ and so $\lambda<8$.
From \cite[Table 2]{PGLSuborbits} we know that $\PGL_2(q)$ acts on the coset space of its maximal subgroup $\DD_{2(q-1)}$ with a subdegree $\frac{q-1}{2}$.
By Lemma \ref{LemmaDavies}, $r\mid\lambda\frac{q-1}{2}$.
So $4\mid\lambda$ and thus $\lambda=4$.
By the basic arithmetical properties in Lemma \ref{LemmaParameters} we have
$(v,k,\lambda)=(\frac{q(q+1)}{2},q+3,4)$ with $b=\frac{q(q+1)(q-1)}{q+3}$.
We get that the only admissible $q$ for $b$ to be a positive integer is $q\in\{5,9\}$.
If $q=5$, then $v=b=15$, $\D$ is symmetric, which contradicts our assumption.
If $q=9$, then $(v,k,\lambda)=(45,12,4)$ with $b=60$, $|G_B|=12$, whereas this is ruled out using {\sc Magma}.

Next, we consider the $\mathcal{C}_3$-subgroup $G_x=\DD_{2(q+1)}$.
Then $$v=\frac{|\PGL_2(q)|}{2(q+1)}=\frac{q(q-1)}{2}.$$
Consider the action of $G_x$ on $\mathcal{D}(x)$.
Similar to the above case, we have $r=\frac{|G_x|}{|G_{xB}|}=\frac{2(q+1)}{|G_{xB}|}$ with some $B$ through $x$, and
$$q(q-1) \leq \frac{1}{2}\lambda q(q-1)=\lambda v<r^2=\frac{4(q+1)^2}{|G_{xB}|^2}.$$
Since we only consider $q \geq 5$, we get $|G_{xB}|\in\{1,2\}$.
If $|G_{xB}|=2$, then we have $\lambda q(q-1)<2(q+1)^2$, which implies that either $\lambda=2$ or $(q,\lambda)\in\{(5,3),(7,3)\}$.
For $(q,\lambda)\in\{(5,3),(7,3)\}$, these do not satisfy the conditions in Lemma \ref{LemmaParameters}.
If $\lambda=2$, then we get $(v,k,\lambda)=(\frac{q(q-1)}{2},q-1,2)$.
Check \cite[Table 1]{AlaviLambda2} and \cite[Theorem 1.1]{MontinaroLambda2}, and we get that $\D$ is a $(10,4,2)$ design with $G=\PGL_2(5)$ and $G_x=\DD_{12}$, as stated.

If $|G_{xB}|=1$, then $r=2(q+1)$. We have
$\lambda q(q-1)<8(q+1)^2$.
In addition, by \cite[Table 2]{PGLSuborbits}, $\PGL_2(q)$ acts on the coset space of its maximal subgroup $\DD_{2(q+1)}$ with a subdegree $\frac{q+1}{2}$.
By Lemma \ref{LemmaDavies}, $r\mid\frac{\lambda(q+1)}{2}$.
This follows that $4\mid\lambda$.
Suppose first that $q\leq7$.
Then  $q\in\{5,7\}$ and it yields $\lambda\in\{4,8,12\}$.
If $(q,\lambda)=(5,4)$ or $(7,4)$, then we get the parameters $(v,k,\lambda)=(10,4,4)$ and $(21,6,4)$. These two cases are ruled out with the aid of  {\sc Magma}.
If $(q,\lambda)=(7,12)$, we get a $2$-$(21,16,12)$ symmetric design, which is out of our consideration. The remaining three cases do not satisfy the conditions in Lemma \ref{LemmaParameters}.
Hence, $q>7$, and we have $\lambda\in\{4,8\}$.

If $\lambda=4$, we get
$$(v,k,\lambda,r,b)=(\frac{q(q-1)}{2},q-1,4,2(q+1),q(q+1)).$$
The block stabilizer $G_B$ has order $\frac{|G|}{b}=q-1$.
By \cite[Theorem 3.5]{MGiudici}, $G_B$ is either a cyclic group $\ZZ_{q-1}$ (conjugate to the Levi complement of the maximal parabolic subgroup),
or conjugate to a subgroup of $S_4$ (when $q$ is a prime satisfying $q\equiv\pm3\pmod{8}$).
For the latter case, $q-1$ divides $24$, so $q\in\{5,13\}$.
The case of $q=5$ is dealt with in the last paragraph.
For $q=13$, we get $(v,k,\lambda)=(78,12,4)$, whereas this is ruled out using {\sc Magma}.
For the former case, $N_G(G_B)$ is the maximal subgroup $\DD_{2(q-1)}$ of type $\mathcal{C}_2$.
By Proposition \ref{LemmaMeet>1}, we know that
$$N_G(G_B)\cap (G_x)^g=N_G(G_B)\cap G_{x^g}>1$$
for each $g\in G$.
So the orbits of $N_G(G_B)$ on $\mathcal{P}$ have length at most $q-1$.
Since $G_B$ has an orbit $B$ of length $q-1$, it follows that $B$ is also an orbit of $N_G(G_B)$. But this implies that $N_G(G_B)=(N_G(G_B))_B=G_B$, a contradiction.

If $\lambda=8$, we get
$$(v,k,\lambda,r,b)=(\frac{q(q-1)}{2},2q-3,8,2(q+1),\frac{q^3-q}{2q-3}).$$
Then $$8b=4q^2+6q+5+\frac{15}{2q-3}.$$
Since $\frac{15}{2q-3}$ is an integer, we have $q\in\{2,3,4,9\}$.
As $q$ is odd, $q\in\{3,9\}$.
If $q=3$, then the socle of $G$ is not simple, not satisfying our assumption.
If $q=9$, then $G=\PGL_2(9)$, and $(v,k,\lambda)=(36,15,8)$, which is ruled out using {\sc Magma}.
\end{proof}

\noindent{\bf Proof of Theorem \ref{ThmMain}}\,\,
Assume that $\D$ is symmetric.
By Lemma \ref{LemmaLambda=1} and its dual statement, since $G_x$ is dihedral, we have $\lambda > 1$.
From Theorem \ref{ThmSym} we know that $G$ is point-imprimitive, a contradiction.
Hence $\D$ is non-symmetric.
Now, it follows from Lemma \ref{LemmaDihedralAffAlmost} that $G$ is either almost simple or affine.
If $G$ is almost simple, then the parts (a)--(e) hold according to Lemmas \ref{LemmaDesignNotSuzuki}--\ref{LemmaDesignqOdd}.
If $G$ is affine, we have the part (f).
\hfill$\square$

\medskip

\section{An application to flag-transitive symmetric designs with point stabilizers of order \texorpdfstring{$2p$}{}}\label{SecProof2}

In this section, we give a classification of flag-transitive symmetric designs with point stabilizers of order $|G_x|=|G_B|=2p$ where $p$ is a prime, as an application of Theorems \ref{ThmSym} and \ref{ThmMain}.

In Corollary \ref{PropoGB=p}, we have treated the case for $|G_x|=|G_B|$ a prime. In this case $G$ must be a Frobenius group.
For the case $|G_x|=|G_B|=2p$, it is shown that $\D$ is a unique $2$-$(16,6,2)$ design and $G$ is imprimitive on both $\P$ and $\B$.

\begin{corollary}\label{ThmGB2p}
Let $G$ be a flag-transitive automorphism group of a symmetric design  $\D$.
Let $x\in\P$ and $B\in\B$.
If $G_x$ has order $2p$ for some prime $p$, then $\D$ is a unique $2$-$(16,6,2)$ design, and $G$ is one of the five flag-regular automorphism groups described in Lemma {\rm\ref{LemmaRegueiroLambda=2}}.
\end{corollary}

\begin{proof}
If $\lambda=1$, by Lemma \ref{LemmaKantorProPlane},then we have either $|G|$ is odd or $|G_x|\ne 2p$. Hence $\lambda>1$.
If $p=2$, then $|G_B|=|G_x|=4$ and $k=4$ (by flag-transitivity). The only possible non-trivial symmetric designs are $2$-$(7,4,2)$ design and $2$-$(13,4,1)$ design.
From Lemmas \ref{LemmaKantorProPlane} and \ref{LemmaLambda=2}, it is known that $G$ is point-primitive and $\D$ is the complement of the point-hyperplane design of $\PG(2,2)$ or the point-hyperplane design of $\PG(2,3)$. While their flag-transitive automorphism groups are $\PSL_3(2)$ and $\PSL_3(3)$, respectively.
None of them have an abelian point stabilizer.

In the following we consider that $p>2$.
Since $|G_B|=|G_x|=2p$, it is clear that $G_B$ is either $\ZZ_{2p}$ or a Frobenius group $\DD_{2p}$.

Now we show that if $\lambda \ne 2$, then $G$ is point-primitive.
Assume that $G$ is point-imprimitive.
Since $k \mid |G_B|$ and $\D$ is non-trivial, we have $k=p$ or $2p$.
If $k=p$, then $(k,\lambda)=1$, and $G$ is point-primitive by \cite[2.3.7]{Dembowski}, a contradiction. Hence $k=2p$. Since $k > \lambda$ and $(k,\lambda) \ne 1$, we have either $\lambda=2x\,(1<x<p)$ or $\lambda=p$.
Then by the analysis in \cite[p.141]{RegueiroReduction}, there exist $n,c$ and $d$ such that
\begin{align}\label{Inequality2p}
\lambda(c-1)=k(d-1)=2p(d-1),
\end{align}
where $v=nc$, $d\mid k$ and $n,c,d,k/d>1$.
If $\lambda=2x$, then by the equation (\ref{Inequality2p}), we get $x(c-1)=p(d-1)$.
If $d=2$, then $x=1$ or $x=p$, a contradiction.
So $d=p$, and $p\mid c-1$ as $p\nmid x$.
Dividing $\lambda(v-1)=r(k-1)$ by equation (\ref{Inequality2p}) and we get
$$\frac{nc-1}{c-1}=\frac{2p-1}{p-1}.$$
It then follows that
$$\frac{c(n-1)}{c-1}=\frac{p}{p-1}.$$
Since $p\mid c-1$, we obtain that $p^2\mid n-1$.
Now $$\frac{p}{p-1}=\frac{c(n-1)}{c-1}\geq\frac{p^2c}{c-1},$$
which is impossible.
So $\lambda=p$.
Note that $G_x$ contains an involution $g$ fixing at least two points $x$ and $y$ (Lemma \ref{LemmaInvolutionFixSym}). Since $k=2p$, we have that $G_B$ is faithful and regular on $B$, and $G_x$ is also faithful and regular on blocks through $x$.
Hence $\langle g \rangle$ acts semi-regularly on the $\lambda$ blocks through two fixed points.
Since $\lambda=p$ is odd, $g$ fixes at least one block $B'$ through $x$ and $y$. Then $1\ne g \in G_{B'} \cap G_x$ and $G_{B'}$ is not regular on $B'$, a contradiction.
Hence, $G$ is point-primitive for $\lambda\ne2$.
Now we analyze the two possibilities.

Case 1.
If $G_B\cong \ZZ_{2p}$, then $G_B^B$ is also abelian and is regular since $G_B^B$ is transitive by Lemma \ref{LemmaFlagtrProperties}.
From Theorem \ref{ThmSym} we know that $G_B$ is faithful on $B$.
So $G_B\cong G_B^B=\ZZ_{2p}$ and $|G_B|=|G_B^B|=k=2p$, implying that $G$ is flag-regular.
If $\lambda=2$, then the result holds by Lemmas \ref{LemmaLambda=2} and \ref{LemmaRegueiroLambda=2}.
If $\lambda\neq2$, then $G$ is point-primitive.
By Theorem \ref{ThmSym}, $G$ is a Frobenius group of odd order, a contradiction.



Case 2. Next, suppose that $G_B$ is a dihedral group $\DD_{2p}$.
By Theorem \ref{ThmSym}, $G_x\cong G_B=D_{2p}$ and $G$ is imprimitive on both $\P$ and $\B$.
It follows that $\lambda=2$, and then by Lemma \ref{LemmaLambda=2}, we get $(v,k,\lambda)=(16,6,2)$.
Hence, $k=6=2p$ and $G$ is flag-regular, and thus the conclusion is drawn by Lemma \ref{LemmaRegueiroLambda=2}.
\end{proof}

\subsection*{Acknowledgements}

The project is supported by the National Natural Science Foundation of China (No.12526634).

\end{document}